\documentclass[12pt]{article}
\usepackage{geometry}
\usepackage[ansinew]{inputenc}
\usepackage{comment}
\usepackage{amssymb}
\usepackage{latexsym}
\usepackage{graphics}
\usepackage{epstopdf}
\usepackage{xcolor}
\usepackage{fancybox}
\usepackage{amsmath,amsfonts}
\begin{document}

\setlength{\parskip}{0.3\baselineskip}

\newtheorem{theorem}{Theorem}
\newtheorem{corollary}[theorem]{Corollary}
\newtheorem{lemma}[theorem]{Lemma}
\newtheorem{proposition}[theorem]{Proposition}
\newtheorem{definition}[theorem]{Definition}
\newtheorem{remark}[theorem]{Remark}
\renewcommand{\thefootnote}{\alph{footnote}}
\newenvironment{proof}{\smallskip \noindent{\bf Proof}: }{\hfill $\Box$\hspace{1in} \medskip \\ }


\newcommand{\beqaa}{\begin{eqnarray}}
\newcommand{\eeqaa}{\end{eqnarray}}
\newcommand{\beqae}{\begin{eqnarray*}}
\newcommand{\eeqae}{\end{eqnarray*}}


\newcommand{\sii}{\Leftrightarrow}
\newcommand{\imer}{\hookrightarrow}
\newcommand{\imerc}{\stackrel{c}{\hookrightarrow}}
\newcommand{\Con}{\longrightarrow}
\newcommand{\con}{\rightarrow}
\newcommand{\conf}{\rightharpoonup}
\newcommand{\confe}{\stackrel{*}{\rightharpoonup}}
\newcommand{\pbrack}[1]{\left( {#1} \right)}
\newcommand{\sbrack}[1]{\left[ {#1} \right]}
\newcommand{\key}[1]{\left\{ {#1} \right\}}
\newcommand{\dual}[2]{\langle{#1},{#2}\rangle}

\newcommand{\R}{{\mathbb R}}
\newcommand{\N}{{\mathbb N}}
\newcommand{\cred}[1]{\textcolor{red}{#1}}

\title{\bf Stability and Regularity the MGT-Fourier Model with Fractional Coupling}
\author{Filomena Barbosa Rodrigues Mendes\\
{\small Department of Engenhary Electric, The  Federal University of Technological of Paran\'a, Brazil}\\Fredy M.  Sobrado  Su\'arez$^*$ \quad and \quad
Santos Richard  W.  Sanguino Bejarano \\
{\small Department of Mathematics,  Federal University of Technological  of Paran\'a, Brazil}
}
\date{}
\maketitle

\let\thefootnote\relax\footnote{$^*$ corresponding author:{\it e-mail:}   {\rm fredy@utfpr.edu.br} (Fredy Maglorio Sobrado  Su\'arez)}.


\begin{abstract}
In this work, we study the stability and regularity of the system formed by the third-order vibration equation in Moore-Gilson-Thompson time
coupled with the classical heat equation with Fourier's law.  We consider fractional couplings. He
the fractional coupling is given by: $\eta A^\phi\theta, \alpha\eta A^\phi u_{tt}$ and $\eta A^\phi u_t$, where the operator $A^\phi$ is self-adjoint and strictly positive
in a complex Hilbert space $H$ and the parameter $\phi$ can vary between $0$
and $1$. When $\phi=1$ we have the MGT-Fourier physical model, previously investigated, see; 2013\cite{ABMvFJRSV2013} and 2022\cite{DellOroPata2022}, in these works the authors respectively showed that the semigroup $S(t) = e^{t\mathbb{B}}$
associated with the MGT-Fourier model are exponentially stable
and analytical.  The model abstract of this research is given by:  \eqref{Eq1.1}--\eqref{Eq1.3}, we show directly that the semigroup $S(t)$ is exponentially stable for $\phi \in [0,1]$, we also show that for $\phi=1$, $S(t)$ is analytic and   study of the Gevrey classes of $S(t)$ and we show that for $\phi\in (\frac{1}{2}, 1)$ there are two families of Gevrey classes: $s_1>2$ when $\phi\in(1/2,2/3]$ and $s_2>\frac{\phi}{2\phi-1}$ when $\phi\in[2/3,1)$,   in the last part of our investigation using spectral analysis we tackled the study of the non-analyticity and lack of Gevrey classes of $S(t)$ when $\phi \in[0,1/2]$. For the study of the existence,  stability, and regularity,  semigroup theory is used together with the techniques of the frequency domain, multipliers, and spectral analysis of system, using proprety  of the fractional operator  $A^\phi$  for $\phi\in[0,1]$.
\end{abstract}

\bigskip
{\sc keyword:} Asymptotic behavior,  Stability, Regularity,   Gevrey Class,  Analyticity,  MGT-Fourier System.

\setcounter{equation}{0}

\section{Introduction}


Various researchers year after year have been devoting their attention to the  study of   asymptotic behaviour  the model  Moore-Gilson-Thompson  equation appearing in the context of acoustic wave propagation in viscous thermally relaxing fluids,    in \cite{BKILasieckaRM2011},  they studied the model abstract:
\begin{equation*}
\tau u_{ttt}+\alpha u_{tt}+c^2 Au+bAu_t.
\end{equation*}   
where $A$ be a selfadjoint positive operator on $\mathbb{H}$ with a dense domain $\mathfrak{D}(A)\subset \mathbb{H}$, show  exponential stability requires  $\gamma\equiv \alpha-\frac{\tau c^2}{b}>0$.  In the complementary region of the
parameters the system is unstable ($\gamma < 0$) or marginally stable ( $\gamma = 0$).  Quite interestingly, it can be used as a model for the vibrations in a standard linear viscoelastic solid, for it can be obtained by differentiating in time the equation of viscoelasticity with an exponential kernel (see \cite{2017OroPata} for more details). 

In \cite{LasieckaWang2015}, studies the decay of the energy of the Moore-Gibson-Thompson (MGT) equation with a viscoelastic term, this work is a generalization of the previous one (Part I) see \cite{LasieckaWang2016},  in the sense that it allows the kernel memory to be more general and shows that power decays the same way kernel memory does, exponentially or not. The model is given by
\begin{equation}
\tau u_{ttt}+\alpha u_{tt}+c^2 Au+bAu_t-\int_0^t g(t-s)Au(s)ds=0,
\end{equation}
where  $A$ is a positive seft-adjoint operator defined in a real Hilbert space $H$. The convolution term $\int_0^t g(t-s)Au(s)ds$ reflects the memory effect of  viscoelastic materials; the ``memory kernel" $g(t)\colon [0,\infty)\to [0,\infty)$ directly relates to whether or how the energy decays. Without this memory term, it is known the MGT equation has exponential energy decay in the non-critical regime, were $\gamma=\alpha-\frac{c^2\tau}{b}>0$, see \cite{BKILasieckaRM2011}.

 One of the first investigations of the coupled MGT-Fourier model, was published in 2013,  see  Alves et al. \cite{ABMvFJRSV2013} studied the exponential decay of the MGT-Fourier physical model given by
\begin{equation*}
\left\{\begin{array}{c}
\alpha u_{ttt}+u_{tt}-a^2\Delta u-a^2\beta\Delta u_t+\eta\Delta \theta=0.\\
\theta_t-\Delta\theta-\alpha\eta\Delta u_{tt}-\eta\Delta u_t=0.
\end{array}\right.
\end{equation*}
where $x\in\Omega$, $t\in (0,\infty)$. The function $u=u(x,t)$ represents the vibration of flexible structures, and $\theta=\theta(x,t)$ is the difference of temperature between the actual state and a reference temperature. The constants  $\alpha$ and $\beta$ are positive, $\eta$ is the coupling constant.  $\Omega$ is a bounded open connected set  in $\mathbb{R}^n(n\geq 1)$ having a smooth boundary $\Gamma=\partial\Omega$.  
The initial conditions are given by \eqref{Eq1.3} and the boundary conditions:
$u=0,  \qquad u_t=0,\qquad \theta=0,\qquad{\rm on}\quad \partial\Omega=\Gamma. $

In more recent research from 2022 \cite{DellOroPata2022} they studied the model  abstract MGT-Fourier system:
\begin{equation*}
\left\{\begin{array}{c}
 u_{ttt}+\alpha u_{tt}+\beta A^\rho u_t+\gamma A^\rho u =\eta A\theta,\\
\theta_t+\kappa A\theta=-\eta Au_{tt}-\eta\alpha Au_t,
\end{array}\right.
\end{equation*}
 in the subcritical regime $\beta>\frac{\gamma}{\alpha}$,  where  the operator $A$ is  strictly positive  selfadjoint.   For any fixed $\rho\in [1,2]$  the authors showed that the solution associated to the semigroup $S(t)$ is analytically and exponentially stable as well. Also in 2022 M. Conti et al \cite{ContiOLP2022}, studied the model:
\begin{equation}\label{Eq1.1Conti2022}
\left\{\begin{array}{ccc}
u_{ttt}+\alpha u_{tt}+\beta\Delta^2 u_t+\gamma\Delta^2 u=-\eta\Delta\theta\\
\theta_t-\kappa\Delta\theta=\eta\Delta u_{tt}+\alpha\eta\Delta u_t.
\end{array}\right.
\end{equation}
   His research is focused on the analysis of the energy transfer between the two equations, particularly when the first one is in the supercritical regime and exhibits an antidissipative character. The main actor then becomes the coupling constant $\eta$, governing the competition between Fourier damping and MGT antidamping. In fact, they showed that a large enough $\eta$ is always capable of stabilizing the system exponentially fast. One of the characteristics of this model is the presence of the bilaplacian in the first equation. With respect to the analogous model with the Laplacian, this forced to adjust the mathematical techniques. On the one hand, the energy estimation method does not seem to be applied directly, on the other hand, there is a gain in regularity that allows us to rely on analytical techniques for the properties of the semigroup associated with the model. In light of the above discussion, the natural question to be addressed is how the dissipation produced by the heat equation influences the asymptotic dynamics of the system. In the subcritical case, both equations in \eqref{Eq1.1Conti2022} are dissipative (actually exponentially stable) and good stabilization properties are expected. This idea was confirmed in \cite{DellOroPata2022}, where it is shown that if $\mu=\gamma-\alpha\beta< 0$ then the semigroup associated to \eqref{Eq1.1Conti2022} is exponentially stable and also analytic. This means that the coupling allows not only a dissipation transfer between the equations, but also a regularity transfer. This research deals with the behavior of the system in the critical and supercritical regimes, that is, when $\mu \geq 0$. In these cases, the system \eqref{Eq1.1Conti2022} consists of a conservative system (if $\mu = 0$) or antidissipative (if $\mu > 0$). More recent research in this direction can be found at \cite{ContiLPata2021, ContiPPQ2020,GamboaNOVVP2017,Quintanilla2019}.
  
 In the last decade, the number of investigations focused on the regularity of the semigroups $S(t)$ associated with coupled systems has increased considerably. The interest is centered on Differentiability,  Gevrey class, and/or Analyticity,  among the model's considered systems with fractional couplings, different types of dissipation (internal, localized), and fractional dissipations. One of the motivations for this new research is that the analytic semigroups associated with linear systems imply that this type of semigroup admits Gevrey classes, which in turn implies the semigroup's differentiability. And depending on the regularity, the semigroup will have the best asymptotic behavior (Exponential Decay).
  More recent research in this direction can be found at \cite{AmmariShelTebou2022,   ContiLPata2022,  DOroJRPata2013,  HPFredy2019, KLiuH2021,  KLiuTebou2022,   LT1998A, Quintanilla2020,  HSLiuRacke2019, BrunaJMR2022,  LTebou2013,   Tebou2021}.

The paper is organized as follows. In section 2, we study the well-posedness of the system \eqref{Eq1.1}-\eqref{Eq1.3} through semigroup theory.  We leave our main contributions for the last two sections, in the third we dedicate to stability, and in the fourth to regularity.     in the section \eqref{3.0} we showed that the semigroup $S(t)=e^{\mathbb{B}t}$ is exponentially stable as $\phi\in[0,1]$.   In the last section \eqref{4.0} of the regularity, we studied the analyticity when $\phi=1$, determine two family of Gevrey classes: $s_1>2$ when $\phi\in (\frac{1}{2},\frac{2}{3}]$ and $s_2>\frac{\phi}{2\phi-1}$ when $\phi\in[\frac{2}{3},1)$. 
And when the parameter $\phi$ assumes values in the interval $[0,\frac{1}{2}]$ it shows that $S(t)$ does not admit Gevrey classes and is not analytic.   We prioritize direct proofs in which we use semigroup-theory characterization together with frequency domain methods, multiplier  and  spectral analysis of the fractional operator  $A^\phi$  for $\phi\in[0,1]$.

\section{ Well-posedness of the system}
In this section, we will use the semigroup theory to assure the existence and uniqueness of strong solutions for the system \eqref{Eq1.1}--\eqref{Eq1.3}  where the operator $A$ is defined to follow, 
let 
$A\colon \mathfrak{D}(A)\subset \mathbf{H}\to \mathbf{H}$ be a strictly positive selfadjoint operator with compact inverse on a complex Hilbert space $\mathbf{H}$,  with a continuous embedded domain $ \mathfrak{D}(A)\hookrightarrow\mathbf{H}$.  Consider  the abstract linear system 
\begin{eqnarray}
\label{Eq1.1}
\alpha u_{ttt}+u_{tt}+a^2A u+a^2\beta A u_t-\eta A^\phi\theta=0,\\
\label{Eq1.2}
\theta_t+A\theta+\alpha\eta A^\phi u_{tt}+\eta A^\phi u_t=0.
\end{eqnarray}
The initial conditions are given by
\begin{equation}
\label{Eq1.3}
u(x,0)=u_0(x),\quad u_t(x,0)=u_1(x),\quad u_{tt}(x,0)=u_2(x),\quad \theta(x,0)=\theta_0(x),\; {\rm in}\; \Omega.
\end{equation}
\begin{remark} It is known that this operator $A$ is strictly positive,   selfadjoint,    has a compact inverse, and has compact resolvent.  And the operator $A^{\tau}$ is self-adjoint positive for all $\tau\in\R$, bounded for $\tau\leq 0$, and  the embedding
\begin{equation*}
\mathfrak{D}(A^{\tau_1})\hookrightarrow \mathfrak{D}(A^{\tau_2}),
\end{equation*}
is continuous for $\tau_1>\tau_2$.  Here,  the norm in $\mathfrak{D}(A^{\tau})$ is given by $\|u\|_{\mathfrak{D}(A^{\tau})}:=\|A^{\tau}u\|$, $u\in \mathfrak{D}(A^{\tau})$, where $\dual{\cdot}{\cdot}$ and $\|\cdot\|$ denotes the inner product and norm in the complex Hilbert space $\mathbf{H}=\mathfrak{D}(A^0)$.  Some of the most used spaces at work are  $\mathfrak{D}(A^\frac{1}{2})$ and $\mathfrak{D}(A^{-\frac{1}{2}})$.
\end{remark}
Taking $v=\alpha u_t+u$, the system \eqref{Eq1.1}-\eqref{Eq1.2} can be rewritten  as
\begin{eqnarray}
\label{Eq1.5}
v_{tt}+a^2A v-a^2(\alpha-\beta)A u_t-\eta A^\phi\theta=0,\\
\label{Eq1.6}
\theta_t+A\theta+\eta A^\phi v_t=0.
\end{eqnarray}
This system \eqref{Eq1.5}-\eqref{Eq1.6}, for $\phi=1$ is the model of a system of the coupled viscoelastic equation coupled with the heat equation given by Fourier's law.
\\
Taking the duality product between equation\eqref{Eq1.5} and  $v_t$,   and \eqref{Eq1.6} with $\theta$,   taking advantage of the self-adjointness of the powers of the operator $A$ and using the identity $v=\alpha u_t+u$,   Let $\beta-\alpha>0$, for every solution of the system \eqref{Eq1.1}-\eqref{Eq1.3}  the total energy $\mathfrak{E}\colon \mathbb{R}^+\to\mathbb{R}^+$ is given in the $t$ by    
\begin{equation}\label{Energia01}
\mathfrak{E}(t)=\frac{1}{2}\bigg[ \|\alpha w+v\|^2+a^2\alpha(\beta-\alpha)\|A^\frac{1}{2}v\|^2
+ a^2\|\alpha A^\frac{1}{2} v+A^\frac{1}{2} u\|^2   +\|\theta\|^2  \bigg ]
\end{equation}
and satisfies
\begin{equation}\label{Dissipa01}
\dfrac{d}{dt}\mathfrak{E}(t)=-a^2(\beta-\alpha)\|A^\frac{1}{2}v\|^2-\|A^\frac{1}{2}\theta\|^2.
\end{equation}
Taking $u_t=v$ and $v_t=w$,   
 the initial boundary value problem \eqref{Eq1.1}-\eqref{Eq1.3} can be reduced to the following abstract initial value problem for a first-order evolution equation
 \begin{equation}\label{Fabstrata}
    \frac{d}{dt}U(t)=\mathbb{B} U(t),\quad    U(0)=U_0,
\end{equation}
 where    $U(t)=(u,v,w,\theta)^T$,  $U_0=(u_0,u_1,u_2,\theta_0)^T$  and   the operator $\mathbb{B}\colon \mathfrak{D}(\mathbb{B})\subset \mathbb{H}\to\mathbb{H}$ is given by
\begin{equation}\label{operadorAgamma}
 \mathbb{B}U:=\Big( v, w,  -\dfrac{1}{\alpha}\big [a^2Au+a^2\beta A v-\eta A^\phi\theta+ w\big ],  -\big[ \eta A^\phi v+\alpha\eta A^\phi w+A \theta\big ] \Big)^T,
\end{equation}
for $U=(u,v,w,\theta)^T$. This operator will be defined in a suitable subspace of the phase space
$$
\mathbb{H}:=[\mathfrak{D}(A^\frac{1}{2})]^2\times[\mathfrak{D}(A^0)]^2,
$$
it is a Hilbert space with the inner product
\begin{eqnarray*}
\dual{ U_1}{U_2}_\mathbb{H} & := & a^2\alpha(\beta-\alpha)\dual{A^\frac{1}{2}v_1}{A^\frac{1}{2}v_2}+a^2 \dual{A^\frac{1}{2}u_1+\alpha A^\frac{1}{2}v_1}{A^\frac{1}{2}u_2+\alpha A^\frac{1}{2}v_2}\\
&  & +\dual{v_1+\alpha w_1}{v_2+\alpha w_2} + \dual{\theta_1}{\theta_2}.
\end{eqnarray*}
for $U_i=(u_i, v_i, w_i,  \theta_i)\in \mathbb{H}$,  $i=1,2$  and induced norm
\begin{equation}\label{NORM}
\|U\|_\mathbb{H}^2:=a^2\alpha(\beta-\alpha)\|A^\frac{1}{2}v\|^2+a^2 \|A^\frac{1}{2}u+\alpha A^\frac{1}{2}v\|^2+\|v+\alpha w\|^2+\|\theta\|^2.
\end{equation}
In these conditions, we define the domain of $\mathbb{B}$ as
\begin{multline}\label{dominioB}
    \mathfrak{D}(\mathbb{B}):= \Big \{ U\in \mathbb{H} \colon  (v,w)\in [\mathfrak{D}(A^\frac{1}{2})]^2,\; a^2u+a^2\beta v-\eta A^{\phi-1}\theta \in \mathfrak{D}(A), \\
    \eta A^{\phi-1} v+\alpha\eta A^{\phi-1}w+\theta \in \mathfrak{D}(A) \Big\}.
\end{multline}
To show that the operator $\mathbb{B}$ is the generator of a $C_0$-semigroup,  we invoke a result from Liu-Zheng' \cite{LiuZ}.

\begin{theorem}[see Theorem 1.2.4 in \cite{LiuZ}] \label{TLiuZ}
Let $\mathbb{B}$ be a linear operator with domain $\mathfrak{D}(\mathbb{B})$ dense in a Hilbert space $\mathbb{H}$. If $\mathbb{B}$ is dissipative and $0\in\rho(\mathbb{B})$, the resolvent set of $\mathbb{B}$, then $\mathbb{B}$ is the generator of a $C_0$-semigroup of contractions on $\mathbb{H}$.
\end{theorem}
\begin{proof}
Let us see that the operator $\mathbb{B}$  given in \eqref{operadorAgamma} satisfies the conditions of this  theorem\eqref{TLiuZ}.  Clearly,  we see that $\mathfrak{D}(\mathbb{B})$ is dense in $\mathbb{H}$. Taking the inner product of $\mathbb{B}U$ with $U$,  we have
\begin{equation}\label{Eqdissipative}
\text{Re}\dual{\mathbb{B}U}{U}_\mathbb{H}= -a^2(\beta-\alpha)\|A^\frac{1}{2}v\|^2 -\|A^\frac{1}{2}\theta \|^2\leq 0, \quad\forall\ U\in \mathfrak{D}(\mathbb{B}),
\end{equation} 
and since $\beta>\alpha$,  the operator  $\mathbb{B}$ is dissipative.

To complete the conditions of the above theorem, it remains to show that $0\in \rho(\mathbb{B})$.  Therefore we must show that $(0I-\mathbb{B})^{-1}$ exists and is bounded in $\mathbb{H}$. We will first prove that $(0I-\mathbb{B})^{-1}$ exists, then it must be proved that $\mathbb{B}$ is bijective.  Here we are going to affirm that $\mathbb{B}$ is surjective,  then for all $F =
(f^1,f^2,f^3,f^4)^T\in \mathbb{H}$   the stationary problem $\mathbb{B} U = F$  has a solution for $ U = (u, v, w,\theta)^T\in \mathfrak{D}(\mathbb{B})$.  From definition of the operator $\mathbb{B}$  in \eqref{operadorAgamma},  this system can be written as:\\
$v  =  f^1,\qquad w  =  f^2$ \qquad  and 
\begin{eqnarray}
\label{Eq02LaxM}
 a^2 Au &= & -\eta A^{\phi-1}f_4-\eta^2A^{2\phi-1}f^1-\alpha\eta^2A^{2\phi-1}f^2 -\alpha f^3-a^2\beta Af^1-f^2,\\
\label{Eq03LaxM}
  A\theta & = & -f^4- \eta A^\phi f^1-\alpha\eta A^\phi f^2.
\end{eqnarray}
From these equations, this problem can be placed in a variational formulation:  $v, w\in [\mathcal{D}(A^\frac{1}{2})]^2$ and we write the last two equations \eqref{Eq02LaxM}-\eqref{Eq03LaxM} in a variational form,  using the  sesquilinear form $b$ :
\begin{equation}\label{Sesquilinear01}
b(u,\theta; \psi_1,\psi_2)=\langle g_1,g_2; \psi_1,\psi_2\rangle
\end{equation}
where $g_1=-\eta A^{\phi-1}f_4-\eta^2A^{2\phi-1}f^1-\alpha\eta^2A^{2\phi-1}f^2 -\alpha f^3-a^2\beta Af^1-f^2,  g_2=-f^4- \eta A^\phi f^1-\alpha\eta A^\phi f^2$ and the sesquilinear form $b$ is given by
\begin{equation}\label{Sesquilinear02}
b(u,\theta;\psi_1,\psi_2)=a^2\dual{A^\frac{1}{2}u}{A^\frac{1}{2}\psi_1}+\dual{A^\frac{1}{2}\theta}{A^\frac{1}{2}\psi_2}.
\end{equation}
as $\mathfrak{D}(A)\hookrightarrow\mathfrak{D}(A^\frac{1}{2})\hookrightarrow D(A^0)\hookrightarrow\mathfrak{D}(A^{-\frac{1}{2}})\hookrightarrow\mathfrak{D}(A^{-1})$,  this sesquilinear form is coercive in the space $[\mathfrak{D}(A^\frac{1}{2})]^2$.  As $g_1,g_2\in [ \mathfrak{D}(A^{-\frac{1}{2}})]^2$ from Lax-Milgram's Lemma the variational form has unique solution $(u,\theta)\in [\mathfrak{D}(A^\frac{1}{2})\times \mathfrak{D}(A^\frac{1}{2})]$     and it satisfies \eqref{Eq02LaxM}-\eqref{Eq03LaxM}, such that \eqref{Sesquilinear01} is verify for $(\psi_1,\psi_2)\in [\mathfrak{D}(A^\frac{1}{2})]^2$.  From \eqref{Sesquilinear02}  and  for all $(\psi_1,\psi_2)\in [\mathfrak{D}(A^\frac{1}{2})]^2$,   we have
\begin{multline}\label{Acotado01}
a^2\dual{A^\frac{1}{2} u}{A^\frac{1}{2}\psi_1}+\dual{A^\frac{1}{2}\theta}{A^\frac{1}{2}\psi_2}=\dual{-f^4- \eta A^\phi f^1-\alpha\eta A^\phi f^2}{\psi_2}\\
+\dual{-\eta A^{\phi-1}f_4-\eta^2A^{2\phi-1}f^1-\alpha\eta^2A^{2\phi-1}f^2 -\alpha f^3-a^2\beta Af^1-f^2}{\psi_1}, 
\end{multline}
choosing $\psi_2=0\in \mathfrak{D}(A^\frac{1}{2})$ from \eqref{Acotado01} and $\forall \psi_1\in\mathfrak{D}(A^\frac{1}{2})$,    we have
\begin{equation*} %
a^2\dual{A^\frac{1}{2} u}{A^\frac{1}{2}\psi_1}=\dual{-\eta A^{\phi-1}f_4-\eta^2A^{2\phi-1}f^1-\alpha\eta^2A^{2\phi-1}f^2 -\alpha f^3-a^2\beta Af^1-f^2}{\psi_1}, 
\end{equation*}
then
\begin{multline}\label{Acotado02}
a^2Au=-\eta A^{\phi-1} f_4-\eta^2A^{2\phi-1}f^1-\alpha\eta^2A^{2\phi-1}f^2 -\alpha f^3\\
-a^2\beta Af^1-f^2\quad{\rm in}\quad  [\mathfrak{D}(A^\frac{1}{2})]^\prime=\mathfrak{D}(A^{-\frac{1}{2}}).
\end{multline}
On the other hand,  choosing $\psi_1=0\in  \mathfrak{D}(A^\frac{1}{2})$ from \eqref{Acotado02} and $\forall \psi_2\in\mathfrak{D}(A^\frac{1}{2})$,    we have
$\dual{A\theta}{\psi_2}=\dual{-f^4- \eta A^\phi f^1-\alpha\eta A^\phi f^2}{\psi_2}$,  then
\begin{equation}\label{Acotado03}
A\theta =-f^4- \eta A^\phi f^1-\alpha\eta A^\phi f^2\quad{\rm in}\quad  [\mathfrak{D}(A^\frac{1}{2})]^\prime=\mathfrak{D}(A^{-\frac{1}{2}}).
\end{equation}
\eqref{Acotado02} and \eqref{Acotado03} are solutions in the weak sense.  From $F=(f^1, f^2, f^3,f^4)^T\in \mathbb{H}$ and  $0\leq\phi\leq 1$  implies that $-\frac{1}{2}\leq \frac{2\phi-1}{2}\leq \frac{1}{2}$,  we  have  $\mathfrak{D}(A^\frac{1}{2})\hookrightarrow\mathfrak{D}(A^\frac{2\phi-1}{2})\hookrightarrow\mathfrak{D}(A^{-\frac{1}{2}})\hookrightarrow\mathfrak{D}(A^\frac{1-2\phi}{2})\hookrightarrow \mathfrak{D}(A^\frac{1}{2})$,   from \eqref{Acotado02},  we have
\begin{multline}\label{Acotado04}
u=\dfrac{1}{a^2}\bigg[ -\eta A^{\phi-2} f_4-\eta^2A^{2\phi-2}f^1-\alpha\eta^2A^{2\phi-2}f^2 -\alpha A^{-1}f^3\\
-a^2\beta f^1-A^{-1}f^2  \bigg] \quad{\rm in}\quad \mathfrak{D}(A^\frac{1}{2}).
\end{multline}
Now  from $0\leq\phi\leq 1$  implies that $2\geq 2-\phi\geq 1\geq \frac{2-\phi}{2}\geq\frac{1}{2}$,  we  have  $\mathfrak{D}(A^\frac{1}{2})\hookrightarrow\mathfrak{D}(A^\frac{2-\phi}{2})\hookrightarrow\mathfrak{D}(A)\hookrightarrow\mathfrak{D}(A^0)\hookrightarrow \mathfrak{D}(A^{\phi-1})\hookrightarrow\mathfrak{D}(A^{-1})$,   from \eqref{Acotado03},  we have
\begin{equation}\label{Acotado05}
\theta = -A^{-1}f^4-\eta A^{\phi-1}f^1-\alpha\eta A^{\phi-1}f^2 \quad{\rm in}\quad \mathfrak{D}(A^\frac{1}{2}).
\end{equation}
then   $a^2u+a^2\beta v-\eta A^{\phi-1}\theta\in \mathfrak{D}(A)$ and $\eta A^{\phi-1}v+\alpha\eta A^{\phi-1}w+\theta\in \mathfrak{D}(A)$. The injectivity of $\mathbb{B}$ follows from the uniqueness given by the Lemma of Lax-Milgram's.    It remains to show that $\mathbb{B}^{-1}$ is a bounded operator.   From  $v=f^1$,  $w=f^2$ equations \eqref{Acotado04} and \eqref{Acotado05} and  \eqref{NORM},  we get
\begin{multline}
\|U\|_\mathbb{H}^2=\alpha a^2(\beta-\alpha)\|A^\frac{1}{2}f^1\|^2+a^2\|\alpha A^\frac{1}{2}f^1+\dfrac{1}{a^2}\big[ -\eta A^{\phi-\frac{3}{2}} f_4-\eta^2A^{2\phi-\frac{3}{2}}f^1-\alpha\eta^2A^{2\phi-\frac{3}{2}}f^2 \\-\alpha A^{-\frac{1}{2}}f^3
-a^2\beta A^\frac{1}{2}f^1-A^{-\frac{1}{2}}f^2  \big]\|^2+\|\alpha f^2+f^1\|^2\\
+\| -A^{-1}f^4-\eta A^{\phi-1}f^1-\alpha\eta A^{\phi-1}f^2\|^2
\end{multline}
Using norm $\|F\|_\mathbb{H}$  and applying  inequalities Cauchy-Schwarz, Young and applying continuous embedding $\mathfrak{D}(A^{\tau_2})\hookrightarrow  \mathfrak{D}(A^{\tau_1}), \;\tau_2 >\tau_1$,   after making various estimates,  we obtain
\begin{equation*}
\|U\|^2_\mathbb{H}=\| \mathbb{B}^{-1}F\|_\mathbb{H}^2\leq C\|F\|^2_\mathbb{H}.
\end{equation*}
Therefore  $\mathbb{B}^{-1}$ is bounded.  So we come to the end of the proof of this theorem.
\end{proof}
 As a consequence of the previous Theorem\eqref{TLiuZ},  we obtain
\begin{theorem}
Given $U_0\in\mathbb{H}$ there exists a unique weak solution $U$ to  the problem \eqref{Fabstrata} satisfying 
$$U\in C([0, +\infty), \mathbb{H}).$$
Futhermore,  if $U_0\in  \mathfrak{D}(\mathbb{B}^k), \; k\in\mathbb{N}$, then the solution $U$ of \eqref{Fabstrata} satisfies
$$U\in \bigcap_{j=0}^kC^{k-j}([0,+\infty),  \mathfrak{D}(\mathbb{B}^j).$$
\end{theorem}
\begin{theorem}[Lions' Interpolation]\label{Lions-Landau-Kolmogorov}  Let $\alpha<\beta<\gamma$. The there exists a constant $L=L(\alpha,\beta,\gamma)$ such that
\begin{equation}\label{ILLK}
\|A^\beta u\|\leq L\|A^\alpha u\|^\frac{\gamma-\beta}{\gamma-\alpha}\cdot \|A^\gamma u\|^\frac{\beta-\alpha}{\gamma-\alpha}
\end{equation}
for every $u\in\mathfrak{D}(A^\gamma)$.
\end{theorem}
\begin{proof}
See  Theorem  5.34 \cite{EN2000}.
\end{proof}
\section{Stability}\label{3.0}
\subsection{Exponential decay of the semigroup $S(t)=e^{t\mathbb{B}}$}
In this section, we will study the asymptotic behavior of the semigroup of the system \eqref{Eq1.1}-\eqref{Eq1.3}.  We will use the following spectral characterization of exponential stability of semigroups due to Gearhart\cite{Gearhart}(Theorem 1.3.2  book of Liu-Zheng \cite{LiuZ}).
\begin{theorem}[see \cite{LiuZ}]\label{LiuZExponential}
Let $S(t)=e^{t\mathbb{B}}$ be  a  $C_0$-semigroup of contractions on  a Hilbert space $ \mathbb{H}$. Then $S(t)$ is exponentially stable if and only if  
	\begin{equation}\label{EImaginario}
\rho(\mathbb{B})\supseteq\{ i\lambda/ \lambda\in \R \} 	\equiv i\R
\end{equation}
and
\begin{equation}\label{Exponential}
 \limsup\limits_{|\lambda|\to
   \infty}   \|(i\lambda I-\mathbb{B})^{-1}\|_{\mathcal{L}( \mathbb{H})}<\infty
\end{equation}
holds.
\end{theorem}
\begin{remark}\label{EqvExponential}
 Note that to show the condition \eqref{Exponential} it is enough to show that: Let $\delta>0$. There exists a constant $C_\delta>0$ such that the solutions of the system \eqref{Eq1.1}-\eqref{Eq1.3} for $|\lambda|>\delta$,  satisfy the inequality 
 \begin{equation}\label{EqvExponencial}
 \dfrac{\|U\|_\mathbb{H}}{\|F\|_\mathbb{H}}\leq C_\delta\quad\Longleftrightarrow\quad \|U\|^2_\mathbb{H}\leq C_\delta\|F\|_\mathbb{H}\|U\|_\mathbb{H}<\infty.
 \end{equation}
 \end{remark}
In view this theorem\eqref{LiuZExponential}, we will try to obtain some estimates  for  the solution $U=(u, v, w, \theta)^T$ of the system $(i\lambda I-\mathbb{B})U=F$,   where  $\lambda\in \R$  and  $F=(f^1, f^2, f^3, f^4)^T\in \mathbb{H}$. This system, written in components,  reads
\begin{eqnarray}
i\lambda u-v &=& f^1\quad {\rm in}\quad \mathfrak{D}(A^\frac{1}{2})\label{Pesp-10}\\
i\lambda v-w &=& f^2\quad {\rm in}\quad \mathfrak{D}(A^\frac{1}{2})\label{Pesp-20}\\
i\alpha\lambda w+a^2Au+a^2\beta Av-\eta A^\phi \theta +w &=& \alpha f^3\quad {\rm in}\quad \mathfrak{D}(A^0) \label{Pesp-30}\\
i\lambda\theta+ \eta A^\phi v+\alpha\eta A^\phi w+A\theta &=& f^4 \quad {\rm in}\quad \mathfrak{D}(A^0) \label{Pesp-40}
\end{eqnarray}
From \eqref{Eqdissipative},   we have the first estimate
\begin{eqnarray*}
a^2(\beta-\alpha)\|A^\frac{1}{2} v\|^2+\|A^\frac{1}{2}\theta \|^2& = & |a^2(\beta-\alpha)\|A^\frac{1}{2} v\|^2+\|A^\frac{1}{2}\theta \|^2|=|-{\rm Re}\dual{\mathbb{B}U}{U}|\\
&= &|{\rm Re}\{i\lambda\|U\|^2-\dual{\mathbb{B}U}{U}\}|=|{\rm Re}\{ \dual{(i\lambda I-\mathbb{B})U}{U}  \}|\\
&\leq & |\dual{F}{U}|\leq \|F\|_\mathbb{H}\|\|U\|_\mathbb{H}.
\end{eqnarray*}
Therefore
\begin{equation}\label{dis-10}
a^2(\beta-\alpha)\|A^\frac{1}{2} v\|^2+\|A^\frac{1}{2}\theta \|^2 \leq  \|F\|_\mathbb{H}\|\|U\|_\mathbb{H}.
\end{equation}

Next, we show some lemmas that will lead us to the proof of the main theorem of this section.
\begin{lemma}\label{Lemma3}
Let $\delta>0$.  There exists $C_\delta>0$ such that the solutions of the system \eqref{Eq1.1}-\eqref{Eq1.3}  for $|\lambda| > \delta$,  satisfy
\begin{eqnarray}\label{Item01Lemma3}
(i)\quad|\lambda|\|A^\frac{1}{2}u\|^2 & \leq & C_\delta\|F\|_\mathbb{H}\|U\|_\mathbb{H}\qquad{\rm for}\qquad  0 \leq\phi \leq 1,\\
\label{Item02Lemma3}
(ii)\quad |\lambda|^2\|A^\frac{1}{2}u\|^2 & \leq & C_\delta\{ \|F\|_\mathbb{H}\|U\|_\mathbb{H}+\|F\|^2_\mathbb{H}\} \qquad{\rm for}\qquad  0 \leq\phi \leq 1,\\
\label{Item03Lemma3}
(iii)\quad |\dual{w}{Au}| & \leq & C_\delta\{ \|F\|_\mathbb{H}\|U\|_\mathbb{H}+\|F\|^2_\mathbb{H}\} \qquad{\rm for}\qquad  0 \leq\phi \leq 1.
\end{eqnarray}
\end{lemma}
\begin{proof}$\!\!(i)$: Taking the duality product between equation\eqref{Pesp-10} and $Au$,  taking advantage of the self-adjointness of the powers of the operator $A$,  we obtain
\begin{equation}\label{Eq01ALemma3}
i\lambda\|A^\frac{1}{2}u\|^2=\dual{A^\frac{1}{2} v}{A^\frac{1}{2}u}+\dual{A^\frac{1}{2}f^1}{A^\frac{1}{2}u}.
\end{equation}
Applying Cauchy-Schwarz and Young inequalities and norms $\|F\|_\mathbb{H}$ and $\|U\|_\mathbb{H}$,  for $\varepsilon>0$, exists $C_\varepsilon>0$ such that
\begin{equation}\label{Eq02Lemma3}
|\lambda|\|A^\frac{1}{2}u\|^2\leq C_\varepsilon \|A^\frac{1}{2}v\|^2
+\varepsilon \|A^\frac{1}{2}u\|^2+C_\delta \|F\|_\mathbb{H} \|U\|_\mathbb{H}.
\end{equation}
Therefore for $|\lambda|>1$, we have $\varepsilon\|A^\frac{1}{2}u\|^2\leq \varepsilon|\lambda|\|A^\frac{1}{2}u\|^2$ and using \eqref{dis-10}  finish proof of item $(i)$  this lemma.\\
{\bf Proof:$(ii)$} On the other hand, multiplying by $\lambda$ the equation \eqref{Eq01ALemma3} and applying Cauchy-Schwarz and Young inequalities,  for $\varepsilon>0$, exists $C_\varepsilon>0$ such that
\begin{equation}
\label{Eq02ALemma3}
|\lambda|^2\|A^\frac{1}{2}u\|^2  \leq  C_\varepsilon \|A^\frac{1}{2}v\|^2
+\varepsilon|\lambda|^2 \|A^\frac{1}{2}u\|^2+C_\varepsilon \|F\|^2_\mathbb{H},
\end{equation}
from \eqref{dis-10}  finish proof of item $(ii)$  this lemma.\\
{\bf Proof:$(iii)$} Using  identity $w=-\lambda^2u-i\lambda f^1-f^2$, we have
\begin{multline*}
\dual{w}{Au}=\dual{-\lambda^2 u-i\lambda f^1-f^2}{Au}=-\lambda^2\|A^\frac{1}{2}u\|^2-i\dual{A^\frac{1}{2}f^1}{\lambda A^\frac{1}{2}u}-\dual{A^\frac{1}{2}f^2}{A^\frac{1}{2}u}.
\end{multline*}
Applying Cauchy-Schwarz  and item $(ii)$ this lema, finish proof this item.
\end{proof}
\begin{theorem}\label{TDExponential}
The semigroup $S(t) = e^{t\mathbb{B}}$   is exponentially stable when the parameter $\phi$ assumes values in the interval $[0,1]$.
\end{theorem}
\begin{proof} 
Let's first check the condition \eqref{EqvExponencial},  we will start by proving in the following lemma the estimate of the term $a^2\|A^\frac{1}{2}u+\alpha A^\frac{1}{2} v\|^2 $  and $\|v+\alpha w\|^2$  of $\|U\|_\mathbb{H}^2$:
\begin{lemma}\label{Lemma6}
Let $\delta>0$.  There exists $C_\delta>0$ such that the solutions of the system \eqref{Eq1.1}-\eqref{Eq1.3}  for $|\lambda| > \delta$,  satisfy
\begin{eqnarray}\label{Eq01Lemma6}
 \|A^\frac{1}{2}u+\alpha A^\frac{1}{2} v\|^2 & \leq &  C_\delta\|F\|_\mathbb{H}\|U\|_\mathbb{H}\qquad{\rm for}\qquad  0 \leq\phi \leq 1.
\end{eqnarray}
\end{lemma}
\begin{proof}
As $\|A^\frac{1}{2}u+\alpha A^\frac{1}{2} v\|^2\leq C [\|A^\frac{1}{2}u\|^2+\|A^\frac{1}{2} v\|^2]$ and for  $|\lambda|>1$ from item $(i)$ of Lemma \ref{Lemma3}    and  estimate  \eqref{dis-10},  we finish the proof of  this lemma.
\end{proof}
\begin{lemma}\label{Lemma6A}
Let $\delta>0$.  There exists $C_\delta>0$ such that the solutions of the system \eqref{Eq1.1}-\eqref{Eq1.3}  for $|\lambda| > \delta$,  satisfy
\begin{equation}\label{Eq01Lemma6A}
\| w \|^2 \leq C_\delta\|F\|_\mathbb{H}\|U\|_\mathbb{H}\qquad{\rm for}\qquad 0 \leq \phi\leq 1.
\end{equation}
\end{lemma}
\begin{proof}
Taking the duality product between equation \eqref{Pesp-20} and $w$ and using \eqref{Pesp-30},  taking advantage of the self-adjointness of the powers of the operator $A$,  we obtain
\begin{multline}\label{Eq02Lemma6A}
 \|w\|^2=-\frac{1}{\alpha}\dual{v}{i\alpha\lambda w}-\dual{f^2}{w}
 = \dfrac{a^2}{\alpha}\dual{A^\frac{1}{2}v}{A^\frac{1}{2}u}\\+\dfrac{a^2\beta}{\alpha}\|A^\frac{1}{2}v\|^2-\dfrac{\eta}{\alpha}\dual{A^\frac{1}{2}v}{A^{\phi-\frac{1}{2}}\theta}+\dfrac{1}{\alpha}\dual{v}{w} -\dual{v}{f^3}-\dfrac{1}{\alpha}\dual{f^2}{w}.
\end{multline}
For $\varepsilon>0$, exists $C_\varepsilon>0$ such that  $\frac{1}{\alpha}|\dual{v}{w}|\leq C_\varepsilon\|v\|+\varepsilon\|w\|^2$. Applying estimates \eqref{dis-10} and  \eqref{Item01Lemma4A} of Lemma \ref{Lemma4A} for $|\lambda|>1$, we finish to proof this is lemma.
\end{proof}
Finally the following lemma estimates the term $\|v+\alpha w\|^2 $ of $\|U\|_\mathbb{H}^2$:
\begin{lemma}\label{Lemma7}
Let $\delta>0$.  There exists $C_\delta>0$ such that the solutions of the system \eqref{Eq1.1}-\eqref{Eq1.3}  for $|\lambda| > \delta$,  satisfy
\begin{equation}\label{Eq01Lemma7}
\|v+\alpha w\|^2 \leq C_\delta\|F\|_\mathbb{H}\|U\|_\mathbb{H}\qquad{\rm for}\qquad  0 \leq\phi \leq 1.
\end{equation}
\end{lemma}
\begin{proof}
As $\|v+\alpha w\|^2\leq C (\|v\|^2+\alpha^2\|w\|^2)^2$ and as $0\leq\frac{1}{2}$ applying continuous embedding,  estimative \eqref{dis-10} and Lemma \ref{Lemma6A},  we finish to proof this is lemma.
\end{proof}
  Finally,  using  the Lemmas \ref{Lemma6}--\ref{Lemma7} and estimates \eqref{dis-10},  we get
\begin{equation}\label{Eq012Exponential}
\|U\|_\mathbb{H}^2\leq C_\delta \|F\|_\mathbb{H}\|U\|_\mathbb{H}\quad{\rm for}\quad 0\leq\phi\leq 1.
\end{equation}
Therefore the condition \eqref{Exponential} for $\phi\in [0,1]$ of Theorem \ref{LiuZExponential} is verified.  Next, we show the condition \eqref{EImaginario} of Theorem \ref{LiuZExponential}.

\begin{lemma}\label{EImaginary}
\label{iR}
Let $\varrho(\mathbb{B})$ be the resolvent set of operator
$\mathbb{B}$. Then
\begin{equation}
i\hspace{0.5pt}\mathbb{R}\subset\varrho(\mathbb{B}).
\end{equation}
\end{lemma}
\begin{proof}
Since $\mathbb{B}$ is a closed operator and $\mathfrak{D}(\mathbb{B})$ has compact embedding into the energy space $\mathbb{H}$, the spectrum $\sigma(\mathbb{B})$ contains only eigenvalues.
 Let us prove that $i\R\subset\rho(\mathbb{B})$ by using an argument by contradiction, so we suppose that $i\R\not\subset \rho(\mathbb{B})$. 
 As $0\in\rho(\mathbb{B})$ and $\rho(\mathbb{B})$ is open, we consider the highest positive number $\lambda_0$ such that the $]-i\lambda_0,i\lambda_0[\subset\rho(\mathbb{B})$ then $i\lambda_0$ or $-i\lambda_0$ is an element of the spectrum $\sigma(\mathbb{B})$.   
 We Suppose $i\lambda_0\in \sigma(\mathbb{B})$ (if $-i\lambda_0\in \sigma(\mathbb{B})$ the proceeding is similar). Then, for $0<\delta<\lambda_0$ there exist a sequence of real numbers $(\lambda_n)$, with $\delta\leq\lambda_n<\lambda_0$, $\lambda_n\con \lambda_0$, and a vector sequence  $U_n=(u_n,v_n,w_n, \theta_n)\in D(\mathbb{B})$ with  unitary norms, such that
\beqae
\|(i\lambda_nI-\mathbb{B}) U_n\|_\mathbb{H}=\|F_n\|_\mathbb{H}\con 0,
\eeqae
as $n\con \infty$.   From  estimative \eqref{Eq012Exponential},   we have 
\begin{multline}
\|U_n\|^2_\mathbb{H}=a^2\alpha(\beta-\alpha)\|A^\frac{1}{2}v_n\|^2 +a^2\|A^\frac{1}{2}u_n+\alpha A^\frac{1}{2}v_n\|^2\\+\|v_n+\alpha w_n\|^2+\|\theta_n\|^2
\leq C_\delta\|F_n\|_\mathbb{H}\|U_n\|_\mathbb{H}=C_\delta\|F_n\|_\mathbb{H}\con 0.
\end{multline}
Therefore, we have  $\|U_n\|_\mathbb{H}\con 0$ but this is absurd, since $\|U_n\|_\mathbb{H}=1$ for all $n\in\N$. Thus, $i\R\subset \rho(\mathbb{B})$. This completes the proof of this lemma. 
\end{proof}
Therefore the semigroup $S(t)=e^{t\mathbb{B}}$ is exponentially stable  for $\phi\in [0,1]$,  thus we finish the proof of this Theorem \ref{TDExponential}. 
\end{proof}
\section{Regularity}
\label{4.0}
\subsection{Study of Analyticity  of the semigroup $S(t)=e^{t\mathbb{B}}$}
In this subsection, we will show that the semigroup $S(t)$, is analytic for the parameter $\phi=1$. The following theorem characterizes the analyticity of the semigroups  $S(t)$.

\begin{theorem}[see \cite{LiuZ}]\label{LiuZAnalyticity}
    Let $S(t)=e^{t\mathbb{B}}$ be $C_0$-semigroup of contractions  on a Hilbert space $\mathbb{H}$.   Suppose that
    \begin{equation}\label{RBEixoIm}
    \rho(\mathbb{B})\supseteq\{ i\lambda/ \lambda\in \R \}  \equiv i\R.
    \end{equation}
     Then $S(t)$ is analytic if and only if
    \begin{equation}\label{Analyticity}
     \limsup\limits_{|\lambda|\to
        \infty}
    \|\lambda(i\lambda I-\mathbb{B})^{-1}\|_{\mathcal{L}(\mathbb{H})}<\infty, 
    \end{equation}
    holds.
\end{theorem}
Before proving the main theorem of this subsection, we will prove some lemmas that will be used for analyticity for $\phi=1$ and the determination of Gevrey classes para $\phi\in(\frac{1}{2},1)$.
\begin{lemma}\label{Lemma4A}
Let $\delta>0$.  There exists $C_\delta>0$ such that the solutions of the system \eqref{Eq1.1}-\eqref{Eq1.3}  for $|\lambda| > \delta$,  satisfy
\begin{eqnarray}\label{Item01Lemma4A}
(i)\quad |\lambda|\|A^\frac{\phi}{2}v\|^2 &\leq& C_\delta\|F\|_\mathbb{H}\|U\|_\mathbb{H}\qquad{\rm for}\qquad  0\leq \phi\leq 1,
\\
\label{Item02Lemma4A}
(ii)\quad |\lambda|^2\|A^\frac{2\phi-1}{2}v\|^2 &\leq& C_\delta\{ \|F\|_\mathbb{H}\|U\|_\mathbb{H}+\|F\|^2_\mathbb{H}\}\quad{\rm for}\quad  0\leq \phi\leq 1.
\end{eqnarray}
\end{lemma}
\begin{proof}\!\!
$(i)$  Taking the duality product between equation \eqref{Pesp-20} and $A^\phi v$,  we obtain
\begin{equation*}
i\lambda\|A^\frac{\phi}{2}v\|^2=\dual{A^\frac{2\phi-1}{2}w}{A^\frac{1}{2}v}+\dual{A^\frac{1}{2}f^2}{A^\frac{2\phi-1}{2}v}.
\end{equation*}
Applying Cauchy-Schwarz and Young  inequalities and estimative \eqref{Item01Lemma4} of Lemma \ref{Lemma4},  we finish the proof of item $(i)$ of this lemma.
\\
{\bf Proof: $(ii)$} Applying the operator $A^\frac{2\phi-1}{2}$ in the equation \eqref{Pesp-20} and then using Young inequality, we will have
\begin{equation*}
|\lambda|^2\| A^\frac{2\phi-1}{2}v\|^2=C_\delta\{ \|A^\frac{2\phi-1}{2}w\|^2+\|A^\frac{2\phi-1}{2}f^2\|^2\}.
\end{equation*}
Using item $(i)$ of the Lemma \ref{Lemma4}, 
we finish the proof   this item.
\end{proof}
\begin{lemma}\label{Lemma4}
Let $\delta>0$.  There exists $C_\delta>0$ such that the solutions of the system \eqref{Eq1.1}-\eqref{Eq1.3}  for $|\lambda| > \delta$,  satisfy
\begin{eqnarray}
\label{Item01Lemma4}
(i)\; \|A^\frac{2\phi-1}{2}w\|^2 &\leq &  C_\delta\|F\|_\mathbb{H}\|U\|_\mathbb{H}\qquad{\rm for}\qquad  0\leq \phi\leq 1,\\
\label{Item02Lemma4}
(ii)\;  |\lambda|^2\|A^{\phi-1}w\|^2 &\leq & C_\delta\{\|F\|_\mathbb{H}\|U\|_\mathbb{H}+\|F\|^2_\mathbb{H}\}\quad{\rm for}\quad  0\leq \phi\leq \dfrac{3}{4},
\\
\label{Item03Lemma4}
(iii)\;  |\lambda|\|A^\frac{\phi-1}{2}w\|^2 &\leq & C_\delta\|F\|_\mathbb{H}\|U\|_\mathbb{H}\quad {\rm for}\quad 0\leq \phi\leq 1,
\\
\label{Item04Lemma4}
(iv)\; |\dual{\lambda \theta}{A^{1-\phi}v}| &\leq  & C_\delta\|F\|_\mathbb{H}\|U\|_\mathbb{H}\quad {\rm for}\quad \dfrac{1}{2}\leq \phi\leq 1.
\end{eqnarray}
\end{lemma}
\begin{proof}
 \!\!$(i)$:
Performing the duality product between equation \eqref{Pesp-40} and $A^{\phi-1}w$,  taking advantage of the self-adjointness of the powers of the operator $A$,  we obtain
\begin{eqnarray}
\nonumber
\eta\alpha\|A^\frac{2\phi-1}{2}w\|^2\hspace*{-0.3cm}&=&\hspace*{-0.3cm}\frac{1}{\alpha}\dual{A^{\phi-1}\theta}{i\alpha\lambda w}-\eta\dual{A^\frac{2\phi-1}{2}v}{A^\frac{2\phi-1}{2}w}-\dual{A^\frac{1}{2}\theta}{A^\frac{2\phi-1}{2}w}+\dual{f^4}{A^{\phi-1}w}\\
\label{Eq02Lemma4}
&= &\hspace*{-0.3cm}-\dfrac{a^2}{\alpha}\dual{A^\frac{2\phi-1}{2}\theta}{A^\frac{1}{2}u}-\dfrac{a^2\beta}{\alpha}\dual{A^\frac{2\phi-1}{2}\theta}{A^\frac{1}{2}v}+\dfrac{\eta}{\alpha}\|A^\frac{2\phi-1}{2}\theta\|^2\\
\nonumber
& &\hspace*{-0.3cm} -\dfrac{1}{\alpha}\dual{A^{-\frac{1}{2}}\theta}{A^\frac{2\phi-1}{2}w}+\dual{A^{\phi-1}\theta}{f^3}-\eta\dual{A^\frac{2\phi-1}{2}v}{A^\frac{2\phi-1}{2}w}\\
\nonumber
& & -\dual{A^\frac{1}{2}\theta}{A^\frac{2\phi-1}{2}w}+\dual{f^4}{A^{\phi-1}w}.
\end{eqnarray}
Applying  Young inequalities and from $-\frac{1}{2}\leq\frac{2\phi-1}{2}\leq\frac{1}{2}$ and $\phi-1\leq 0$,   using continuous embedding,  for $\varepsilon>0$,  exists $C_\varepsilon>0$ such that
\begin{multline}\label{Eq03Lemma4}
\|A^\frac{2\phi-1}{2}w\|^2\leq C\{ \|A^\frac{1}{2}v\|^2+\|A^\frac{1}{2}u\|^2+\|A^\frac{1}{2}\theta\|^2\}+C_\varepsilon\{ \|A^\frac{1}{2}v\|^2+\|A^\frac{1}{2}\theta\|^2\}\\
+\varepsilon\|A^\frac{2\phi-1}{2}w\|^2+C\|\theta\|\|f^3\|+C\|f^4\|\|w\|.
\end{multline}
For $|\lambda|>1$, we have $\|A^\frac{1}{2}u\|^2\leq |\lambda|\|A^\frac{1}{2}u\|^2$, from Lemma \ref{Lemma3},  estimative \eqref{dis-10},   using norms $\|F\|_\mathbb{H}$ and $\|U\|_\mathbb{H}$  finish proof of item $(i)$ this lemma.
\\
{\bf Proof $(ii)$:}  On the other hand, performing the duality product between equation \eqref{Pesp-30} and $\lambda A^{2(\phi-1)}w$,  taking advantage of the self-adjointness of the powers of the operator $A$,  we obtain
\begin{eqnarray}
\nonumber
i\alpha|\lambda|^2\|A^{\phi-1}w\|^2 & = & -a^2\dual{\lambda A^\frac{2\phi-1}{2}u}{A^\frac{2\phi-1}{2}w}-a^2\beta\dual{\lambda A^\frac{2\phi-1}{2}v}{A^\frac{2\phi-1}{2}w} -\lambda\|A^{\phi-1}w\|^2\\
\nonumber
& &+\eta\dual{\lambda\theta}{A^{3\phi-2}w}+\alpha\dual{f^3}{\lambda A^{2(\phi-1)}w}\\
\nonumber
&=& -a^2\dual{\lambda A^\frac{2\phi-1}{2}u}{A^\frac{2\phi-1}{2}w}+ia^2\beta  \|A^\frac{2\phi-1}{2}w\|^2+\alpha\dual{f^3}{\lambda A^{2(\phi-1)}w}\\
\nonumber
& & +ia^2\beta\dual{A^\frac{1}{2}f^2}{A^{2\phi-\frac{3}{2}}w}+\eta\dual{\lambda\theta}{A^{3\phi-2}w}-\lambda\|A^{\phi-1}w\|^2,
\end{eqnarray}
then, taking imaginary part, for $\varepsilon>0$ exists $C_\varepsilon>0$, such that 
\begin{eqnarray}
\label{Eq01ALemma4}
|\lambda|^2 \|A^{\phi-1}w\|^2 & \leq &  C_\delta\{ |\lambda |^2 \|A^\frac{2\phi-1}{2}u\|^2+\|A^\frac{2\phi-1}{2}w\|^2+\|A^\frac{1}{2}f^2\| \|A^{2\phi-\frac{3}{2}}w\|\\
\nonumber
& & +C_\varepsilon|\|A^\frac{1}{2}\theta\|^2+\varepsilon|\lambda|^2\|A^{3\phi-\frac{5}{2}}w\|^2+C_\varepsilon\|f^3\|^2+
\varepsilon|\lambda|^2\|A^{2(\phi-1)}w\|^2\}.
\end{eqnarray}
Therefore,  as  $0\leq\phi\leq\frac{3}{4}\quad\Rightarrow \quad 0\leq\frac{2\phi-1}{2}\leq\frac{1}{4}<\frac{1}{2}$,  of item $(ii)$ Lemma \ref{Lemma3}, item $(i)$ this lemma and estimative \eqref{dis-10}, finish proof this item.\\
{\bf Proof $(iii)$:}  Multiplying by $A^{\phi-1}w$ the equation \eqref{Pesp-30} and by $\eta A^{2\phi-2}w$ the equation \eqref{Pesp-40} and adding the results, we have
\begin{multline*}
i\alpha\lambda\|A^\frac{\phi-1}{2}w\|^2+a^2\dual{A^\phi u}{i\lambda v-f^2}+a^2\beta\dual{A^\phi v}{i\lambda v-f^2}-\eta\dual{A^\phi\theta}{A^{\phi-1}w}\\
+\|A^\frac{\phi-1}{2}w\|^2+i\eta\dual{\lambda\theta}{A^{2\phi-2}w}+\eta^2\dual{A^\phi v}{i\lambda A^{2\phi-2}v-A^{2\phi-2}f^2}\\
+\alpha\eta^2\|A^\frac{3\phi-2}{2}w\|^2+\eta\dual{A^\phi\theta}{A^{\phi-1}w}=\alpha\dual{f^3}{A^{\phi-1}w}+\eta\dual{f^4}{A^{2\phi-2}w}.
\end{multline*}
Taking imaginary part, we have
\begin{multline}
\label{eq002v}
\alpha\lambda\|A^\frac{\phi-1}{2}w\|^2=a^2\beta\lambda\|A^\frac{\phi}{2}v\|^2+\eta^2\lambda\|A^\frac{3\phi-2}{2}v\|^2 +{\rm Im}\{a^2\beta\dual{A^\frac{2\phi-1}{2}v}{A^\frac{1}{2}f^2}\\
+ia^2\dual{\sqrt{|\lambda|}A^\frac{\phi}{2}u}{\dfrac{\lambda}{\sqrt{|\lambda|}} A^\frac{\phi}{2}v}
+a^2\dual{A^\frac{2\phi-1}{2}u}{A^\frac{1}{2} f^2}+ \eta^2\dual{A^{3\phi-\frac{5}{2}}v}{A^\frac{1}{2}f^2} \\
+i\eta\dual{\dfrac{\lambda}{\sqrt{|\lambda|}}\theta}{\sqrt{|\lambda|}A^{2\phi-2}w} +\alpha\dual{f^3}{A^{\phi-1}w}+\eta\dual{f^4}{A^{2\phi-2}w}\}.
\end{multline}
Of \eqref{eq002v}, as for $0\leq\phi\leq 1$ we have $\frac{3\phi-2}{2}\leq \frac{2\phi-1}{2}\leq \frac{\phi}{2}\leq\frac{1}{2}$  and $3\phi-\frac{5}{2}\leq \frac{1}{2}$,  using continuous embedding and  aplying Cauchy-Schwarz and Young inequalities  for $\varepsilon>0$ exists $C_\varepsilon>0$,  such that
\begin{multline*}
|\lambda|\|A^\frac{\phi-1}{2}w\|^2\leq C_\delta\{ \|F\|_\mathbb{H}\|U\|_\mathbb{H}+|\lambda|\|A^\frac{\phi}{2}v\|^2+|\lambda|\|A^\frac{1}{2}u\|^2\}\\
+C_\varepsilon|\lambda| \|\theta\|^2+\varepsilon|\lambda|\|A^{2\phi-2}w\|^2\quad{\rm for}\quad 0\leq\phi\leq 1.
\end{multline*}
Finally, as  $2\phi-2\leq \frac{\phi-1}{2}$ using continuous embedding and    item $(i)$ Lemma \ref{Lemma3} and item $(i)$ of  Lemma \ref{Lemma4A}, finish proof this item.
\\
{\bf Proof  $(iv)$:} Just observe that for $\frac{1}{2}\leq\phi$, we have $\frac{1}{2}-\phi\leq \frac{2\phi-1}{2}\leq\frac{1}{2}$. Then
$$|\dual{\lambda \theta}{A^{1-\phi}v}|=|\dual{A^\frac{1}{2}\theta}{A^{\frac{1}{2}-\phi}(-iw-if^2)}|\leq |\dual{A^\frac{1}{2}\theta}{A^{\frac{1}{2}-\phi}w}|+|\dual{A^\frac{1}{2}\theta}{A^{\frac{1}{2}-\phi}f^2}|.$$
\end{proof}
\begin{lemma}\label{Lemma5A}
Let $\delta>0$.  There exists $C_\delta>0$ such that the solutions of the system \eqref{Eq1.1}-\eqref{Eq1.3}  for $|\lambda| > \delta$,  satisfy
\begin{eqnarray}
\label{Item01Lemma5A}
  \hspace*{-0.25cm}(i)\;|\lambda|\|\theta\|^2\hspace*{-0.25cm} &\leq & \hspace*{-0.25cm}C_\delta\|F\|_\mathbb{H}\|U\|_\mathbb{H}\qquad{\rm for}\quad 0\leq \phi\leq 1,
 \\
 \label{Item03Lemma5A}
 \hspace*{-0.25cm}(ii)\; |\lambda|^2\|\theta\|^2 \hspace*{-0.25cm}&\leq &\hspace*{-0.25cm} C_\delta\{\|F\|_\mathbb{H}\|U\|_\mathbb{H}+\|F\|^2_\mathbb{H}\}\quad{\rm for}\quad 0\leq \phi\leq \dfrac{3}{4},\\
 \label{Item04Lemma54}
 \hspace*{-0.25cm}(iii)\; |\lambda|^2\|A^\frac{2\phi-3}{2}\theta\|^2 & \leq & C_\delta\{\|F\|_\mathbb{H}\|U\|_\mathbb{H}+\|F\|^2_\mathbb{H}\}\quad{\rm for}\quad 0\leq \phi\leq 1.
\end{eqnarray}
\end{lemma}
\begin{proof}
$(i)$:  Performing the duality product between equation \eqref{Pesp-40} and $\lambda A^{-1}\theta$,  taking advantage of the self-adjointness of the powers of the operator $A$,  we obtain
\begin{eqnarray}
\nonumber
\lambda\|\theta\|^2 &=& -i\lambda^2\|A^{-\frac{1}{2}}\theta\|^2-\eta\dual{-iw-if^2}{A^{\phi-1}\theta}-\eta\dual{\alpha\lambda w}{A^{\phi-1}\theta}+\dual{f^4}{\lambda A^{-1}\theta}\\
\nonumber
&=&-i\lambda^2\|A^{-\frac{1}{2}}\theta\|^2+i\eta\dual{w}{A^{\phi-1}\theta}+i\eta\dual{f^2}{A^{\phi-1}\theta}-i\eta a^2\dual{A^\frac{1}{2}u}{A^{\phi-\frac{1}{2}}\theta}\\
\nonumber
& & -i\eta a^2\beta\dual{A^\frac{1}{2}v}{A^{\phi-\frac{1}{2}} \theta}+i\eta^2\|A^\frac{2\phi-1}{2}\theta\|^2-i\eta\dual{w}{A^{\phi-1}\theta}+i\eta\alpha\dual{f^3}{A^{\phi-1}\theta}\\
\label{Eq02Lemma5A}
& & +i\eta\dual{f^4}{A^{\phi-1}v}+i\alpha\eta\dual{f^4}{A^{\phi-1}w}+i\dual{f^4}{\theta}-i\|A^{-\frac{1}{2}}f^4\|^2.
\end{eqnarray}
Taking the real part of \eqref{Eq02Lemma5A} and simplifying, we get
\begin{eqnarray}
\nonumber
\lambda\|\theta\|^2 &=&{\rm Re}\{ i\eta\dual{f^2}{A^{\phi-1}\theta}-i\eta a^2\dual{A^\frac{1}{2}u}{A^{\phi-\frac{1}{2}}\theta} -i\eta a^2\beta\dual{A^\frac{1}{2}v}{A^{\phi-\frac{1}{2}} \theta}\\
\label{Eq03Lemma5A}
& &+i\eta\alpha\dual{f^3}{A^{\phi-1}\theta} +i\eta\dual{f^4}{A^{\phi-1}v}+i\alpha\eta\dual{f^4}{A^{\phi-1}w}+i\dual{f^4}{\theta}\}.
\end{eqnarray}
As $|{\rm Re}\{-i\eta a^2\dual{A^\frac{1}{2}u}{A^{\phi-\frac{1}{2}}\theta} \}|\leq C\{ \|A^\frac{1}{2}u\|^2+\|A^{\phi-\frac{1}{2}}\theta\|^2\}$.    And for  $|\lambda|>1$ and  $0\leq\phi\leq 1$ we have,   $\phi-1\leq 0$  and  $\phi-\frac{1}{2}\leq \frac{1}{2}$,  applying continuous embedding $\mathfrak{D}(A^{\tau_2})\hookrightarrow  \mathfrak{D}(A^{\tau_1}), \;\tau_2 >\tau_1$,     estimative  \eqref{dis-10} and applying Cauchy-Schwarz,  we finish the proof of  item $(i)$ of this lemma.
\\
 {\bf Proof of $(ii)$:}   Multiplying by $\lambda$ the equation  \eqref{Eq03Lemma5A}, we have
 \begin{eqnarray}
\nonumber
|\lambda|^2\|\theta\|^2\hspace*{-0.25cm} &= &\hspace*{-0.25cm}{\rm Re}\{ i\eta\dual{A^{\phi-1}f^2}{\lambda\theta}-i\eta a^2\dual{\lambda A^\frac{1}{2}u}{A^{\phi-\frac{1}{2}}\theta} -\eta a^2\beta\dual{i\lambda v}{A^\phi\theta}+i\dual{f^4}{\lambda \theta}\\
\label{Eq04Lemma5A}
& &+i\eta\alpha\dual{f^3}{\lambda A^{\phi-1}\theta} +i\eta\dual{f^4}{\lambda A^{\phi-1}v}+i\alpha\eta\dual{f^4}{\lambda A^{\phi-1}w}.
\end{eqnarray}
For $\varepsilon>0$, exists $C_\varepsilon>0$,  such that 
 \begin{eqnarray}
 \label{Eq05Lemma5A}
 |{\rm Re}\{ i\eta\dual{A^{\phi-1}f^2}{\lambda\theta}\}|\leq \varepsilon|\lambda|^2\|\theta\|^2+C_\varepsilon\|A^{\phi-1}f^2\|^2,\\
 \label{Eq06Lemma5A}
 |{\rm Re}\{ -i\eta a^2\dual{\lambda A^\frac{1}{2}u}{A^{\phi-\frac{1}{2}}\theta}|\leq C_\delta\{|\lambda|^2\|A^\frac{1}{2}u\|^2+\|A^{\phi-\frac{1}{2}}\theta\|^2\},
 \end{eqnarray}
 on the other hand, from equation \eqref{Pesp-20}, we have $i\lambda v=w+f^2$ and using Lemma \ref{Lemma4}, we  have
 \begin{eqnarray}
 \nonumber
  |{\rm Re}\{-\eta a^2\beta\dual{i\lambda v}{A^\phi\theta}\}| 
  &= & \eta a^2\beta |{\rm Re}\{\dual{A^\frac{2\phi-1}{2}w}{A^\frac{1}{2}\theta}+\dual{A^\frac{1}{2}f^2}{A^{\phi-\frac{1}{2}}\theta}\}|\\
  \label{Eq07Lemma5A}
  &\leq & C_\delta\{ \|F\|_\mathbb{H}\|U\|_\mathbb{H}+\|F\|^2_\mathbb{H}\},
 \end{eqnarray}
besides, 
\begin{eqnarray}
\label{Eq08Lemma5A}
|{\rm Re}\{i\alpha\eta\dual{f^3}{\lambda A^{\phi-1}\theta}\}|\leq C_\varepsilon\|f^3\|^2+\varepsilon |\lambda|^2\|A^{\phi-1}\theta\|^2,
\end{eqnarray}
using equation \eqref{Pesp-20}, we have 
\begin{eqnarray}
\nonumber
|{\rm Re}\{-\eta\dual{f^4}{A^{\phi-1}i\lambda v}\}| & = &  |{\rm Re}\{-\eta\dual{f^4}{ A^{\phi-1}w+A^{\phi-1}f^2}\}|\\
\nonumber
  &\leq &  C_\delta\{ \|f^4\|\|A^{\phi-1}w\|+\|f^4\|\|A^{\phi-1}f^2\|\}\\
  \label{Eq09Lemma5A}
  &\leq &  C_\delta\{ \|F\|_\mathbb{H}\|U\|_\mathbb{H}+\|F\|^2_\mathbb{H}\}
\end{eqnarray}
 and 
 \begin{eqnarray}
 \label{Eq10Lemma5A}
 |{\rm Re}\{ -\alpha\eta\dual{f^4}{A^{\phi-1}i\lambda w}\}|\leq C_\varepsilon\|F\|^2+\varepsilon|\lambda|^2\|A^{\phi-1}w\|^2\qquad{\rm for}\quad 0\leq\phi\leq 1,
 \end{eqnarray}
 or for $0\leq\phi\leq\frac{3}{4}$ of item $(ii)$ Lemma \ref{Lemma4}
  \begin{eqnarray}
  \nonumber
 |{\rm Re}\{ -\alpha\eta\dual{f^4}{A^{\phi-1}i\lambda w}\}| &\leq &  C_\varepsilon\|F\|^2_\mathbb{H}+\varepsilon\{\|F\|_\mathbb{H}\|U\|_\mathbb{H}+\|F\|^2_\mathbb{H}\}\\
   \label{Eq11Lemma5A}
 & \leq & C_\delta\{\|F\|_\mathbb{H}\|U\|_\mathbb{H}+\|F\|^2_\mathbb{H}\}.
 \end{eqnarray}
 And besides, for $0\leq\phi\leq 1$
 \begin{equation}
 \label{Eq12Lemma5A}
  |{\rm Re} \,i\dual{f^4}{\lambda\theta}|\leq C_\varepsilon\|F\|^2_\mathbb{H}+\varepsilon|\lambda|^2\|\theta\|^2.
\end{equation}  
Therefore, from  the estimates \eqref{Eq05Lemma5A}--\eqref{Eq09Lemma5A} together with the estimates \eqref{Eq11Lemma5A} and  \eqref{Eq12Lemma5A} we conclude the proof of the item $(ii)$ of this lemma.
\\
{\bf Proof.} $(iii)$: Performing the duality product between equation \eqref{Pesp-40} and $\lambda A^{2\phi-3}\theta$,  taking advantage of the self-adjointness of the powers of the operator $A$ and using \eqref{Pesp-20},  we obtain
\begin{multline}
\label{Eq13Lemma5A}
 i |\lambda|^2\|A^\frac{2\phi-3}{2}\theta\|^2  =  -\eta\dual{-iw-if^2}{A^{3\phi-3}\theta}-\alpha\eta\dual{A^\frac{2\phi-1}{2}w}{\lambda A^{2\phi-\frac{5}{2}}\theta}\\
  -\lambda\|A^{\phi-1}\theta\|^2+\dual{f^4}{\lambda A^{2\phi-3}\theta}.
 \end{multline}
 As for $0\leq\phi\leq 1$  we have  $3\theta-3\leq 0<\frac{1}{2}, \; 2\phi-\frac{5}{2}\leq \frac{2\phi-3}{2}$ and $2\phi-3\leq\frac{2\phi-3}{2}$, taking imaginary part in \eqref{Eq13Lemma5A} and using continuous immersions,  for $\varepsilon>0$ exists $C_\varepsilon>0$ such that
 \begin{multline}
 |\lambda|^2\|A^\frac{2\phi-3}{2}\theta\|^2 \leq C_\delta \{ \| w\|^2+\|A^\frac{1}{2}\theta \|^2+\|f^2\| \|\theta\|\}+\varepsilon |\lambda|^2 \|A^\frac{2\phi-3}{2}\theta\|^2
 \\
 +C_\varepsilon \|A^\frac{2\phi-1}{2}w\|^2+C_\varepsilon\|f^4\|^2.
 \end{multline}
 Finally  using estimative \eqref{dis-10} and  Lemma \ref{Lemma3}, we finish the proof of the item $(iv)$ of this Lemma.
\end{proof}
\begin{lemma}\label{Lemma4N}
Let $\delta>0$.  There exists $C_\delta>0$ such that the solutions of the system \eqref{Eq1.1}-\eqref{Eq1.3}  for $|\lambda| > \delta$,  satisfy
\begin{eqnarray}
\label{Item01Lemma4N}
(i)\quad |\lambda|\|w\|^2 &\leq &  C_\delta\{ \|F\|_\mathbb{H}\|U\|_\mathbb{H}+\|F\|^2_\mathbb{H}\}\quad{\rm for}\quad   \dfrac{1}{2}\leq\phi\leq 1,
\\
\label{Item02Lemma4N}
(ii)\quad 
|\lambda|\|A^\frac{1}{2}v\|^2 &\leq & C_\delta \{ |\lambda| \|\theta\|^2 +|\lambda|\|w\|^2+\|F\|_\mathbb{H}\|U\|_\mathbb{H}\}\\
\nonumber
& & +C_\varepsilon|\lambda| \|A^\frac{1}{2}u\|^2+\varepsilon|\lambda|\|A^\frac{1}{2}v\|^2 \quad{\rm for}\quad \dfrac{1}{2}\leq\phi\leq 1.
\end{eqnarray}
\end{lemma}
\begin{proof}
$(i)$: Multiplying by $w$ equation \eqref{Pesp-30} and by $\eta A^{\phi-1}w$ the equation \eqref{Pesp-40} and adding the results, we have
\begin{multline}
\label{Eq000vA}
i\alpha\lambda\|w\|^2+a^2\dual{Au}{i\lambda v-f^2}+a^2\beta\dual{Av}{w}-\eta\dual{A^\phi\theta}{w}\\
+\|w\|^2+i\eta\dual{\lambda\theta}{A^{\phi-1}w}+\eta^2\dual{A^\phi v}{i\lambda A^{\phi-1}v-A^{\phi-1}f^2}\\
+\alpha\eta^2\|A^\frac{2\phi-1}{2}w\|^2+\eta\dual{A^\phi\theta}{w}=\alpha\dual{f^3}{w}+\eta\dual{f^4}{A^{\phi-1}w}.
\end{multline}
On the other hand. Carrying out the following duality products: first between equation \eqref{Pesp-40} and $\frac{a^2\beta}{\alpha\eta} A^{1-\phi}v$, in the sequence between equation \eqref{Pesp-30} and $\frac{1}{\alpha\eta}A^{1-\phi}\theta$ and equation \eqref{Pesp-40} with
$\frac{a^2}{\alpha\eta}A^{1-\phi}u$, taking advantage of the self-adjunction of the powers of operator A, we have
\begin{gather}
\label{Eq001v}
a^2\beta\dual{w}{Av}+i\dfrac{a^2\beta}{\alpha\eta}\dual{\lambda\theta}{A^{1-\phi}v}+\dfrac{a^2\beta}{\alpha}\|A^\frac{1}{2}v\|^2+\dfrac{a^2\beta}{\alpha\eta}\dual{A\theta}{A^{1-\phi}v}=\dfrac{a^2\beta}{\alpha\eta}\dual{f^4}{A^{1-\phi}v}\\
\nonumber
\dfrac{a^2\beta}{\alpha\eta}\dual{A^{1-\phi}v}{A\theta}+i\dfrac{1}{\eta}\dual{\lambda w}{A^{1-\phi}\theta}+\dfrac{a^2}{\alpha\eta}\dual{Au}{A^{1-\phi}\theta}-\dfrac{1}{\alpha}\|A^\frac{1}{2}\theta\|^2\\
\label{Eq002v}
+\dfrac{1}{\alpha\eta}\dual{w}{A^{1-\phi}\theta}=\dfrac{1}{\eta}\dual{f^3}{A^{1-\phi}\theta}\\
\nonumber
\dfrac{a^2}{\alpha\eta}\dual{A^{1-\phi}\theta}{Au} +i\dfrac{a^2}{\alpha\eta}\dual{\lambda\theta}{A^{1-\phi}u}+\dfrac{a^2}{\alpha}\dual{A^\frac{1}{2}v}{A^\frac{1}{2}u}+a^2\dual{i\lambda v-f^2}{Au}\\
\label{Eq003v}
=\dfrac{a^2}{\alpha\eta}\dual{f^4}{A^{1-\phi}u}.
\end{gather}
Adding the equations \eqref{Eq000vA}, \eqref{Eq001v}, \eqref{Eq002v} and \eqref{Eq003v}, we get
\begin{eqnarray}
\nonumber
i\alpha\lambda\|w\|^2 & = & a^2\lambda i\{ \dual{ A^\frac{1}{2}u}{A^\frac{1}{2}v} - \dual{A^\frac{1}{2}v}{A^\frac{1}{2}u}\}+a^2\dual{A^\frac{1}{2}u}{A^\frac{1}{2}f^2}-\|w\|^2\\
\nonumber
& & -a^2\beta\{\dual{Av}{w}+\dual{w}{Av}\}-i\eta\dual{\dfrac{\lambda}{\sqrt{|\lambda|}}A^\frac{\phi-1}{2}\theta}{\sqrt{|\lambda|}A^\frac{\phi-1}{2}w}
\\
\nonumber
& & +i\eta^2\lambda\|A^\frac{2\phi-1}{2}v\|^2 +\eta^2\dual{A^{2\phi-\frac{3}{2}}v}{A^\frac{1}{2}f^2}-\alpha\eta^2\|A^\frac{2\phi-1}{2}w\|^2\\
\label{Eq001Ai}
& & +\eta\dual{f^4}{A^{\phi-1}w}
-i\dfrac{a^2\beta}{\alpha\eta}\dual{\lambda \theta}{A^{1-\phi}v}-\dfrac{a^2\beta}{\alpha}\|A^\frac{1}{2}v\|^2\\
\nonumber
& & +\alpha\dual{f^3}{w}-\dfrac{a^2\beta}{\alpha\eta}\{\dual{A\theta}{A^{1-\phi}v}+\dual{A^{1-\phi}v}{A\theta}\}+\dfrac{a^2\beta}{\alpha\eta}\dual{f^4}{A^{1-\phi}v}\\
\nonumber
& & -i\dfrac{1}{\eta}\dual{\lambda w}{A^{1-\phi}\theta}-\dfrac{a^2}{\alpha\eta}\{\dual{Au}{A^{1-\phi}\theta}+\dual{A^{1-\phi}\theta}{Au}\}+\dfrac{1}{\alpha}\|A^\frac{1}{2}\theta\|^2\\
\nonumber
& & -\dfrac{1}{\alpha\eta}\dual{A^\frac{2\phi-1}{2}w}{A^{\frac{3}{2}-2\phi}\theta}+\dfrac{1}{\eta}\dual{f^3}{A^{1-\phi}\theta}-i\dfrac{a^2}{\alpha\eta}\dual{\dfrac{\lambda}{\sqrt{|\lambda|}}A^{\frac{1}{2}-\phi}\theta}{\sqrt{|\lambda|}A^\frac{1}{2}u}\\
\nonumber
& & -\dfrac{a^2}{\alpha}\dual{A^\frac{1}{2}v}{A^\frac{1}{2}u}+a^2\dual{A^\frac{1}{2}f^2}{A^\frac{1}{2}u}+\dfrac{a^2}{\alpha\eta}\dual{f^4}{A^{1-\phi}u}.
\end{eqnarray}
From the identities: $\forall z \in \mathbb{C} $, it is verified $ {\rm Im} \{z+\overline{z} \} = 0 $ and $ {\rm Im} \{i(z- \overline{z}) \} = 0$, taking the imaginary part in \eqref{Eq001Ai}, we get
\begin{eqnarray}
\nonumber
\alpha\lambda\|w\|^2 & = & {\rm Im}\{a^2\dual{A^\frac{1}{2}u}{A^\frac{1}{2}f^2}-i\eta\dual{\dfrac{\lambda}{\sqrt{|\lambda|}}A^\frac{\phi-1}{2}\theta}{\sqrt{|\lambda|}A^\frac{\phi-1}{2}w}\}+\eta^2\lambda\|A^\frac{2\phi-1}{2}v\|^2
\\
\nonumber
& &+ {\rm  Im}\{ \eta^2\dual{A^{2\phi-\frac{3}{2}}v}{A^\frac{1}{2}f^2}-i\dfrac{a^2\beta}{\alpha\eta}\dual{\lambda\theta}{\lambda A^{1-\phi}v}+\alpha\dual{f^3}{w}\\
\label{Eq002Ai}
& & +\eta\dual{f^4}{A^{\phi-1}w}+\dfrac{a^2\beta}{\alpha\eta}\dual{f^4}{A^{1-\phi}v} -\dfrac{1}{\alpha}\dual{A^\frac{2\phi-1}{2}w}{A^{\frac{3}{2}-2\phi}\theta}\\
\nonumber
& &-i\dfrac{1}{\eta}\dual{\lambda w}{A^{1-\phi}\theta}+\dfrac{1}{\eta}\dual{f^3}{A^{1-\phi}\theta} -i\dfrac{a^2}{\alpha\eta}\dual{\dfrac{\lambda}{\sqrt{|\lambda|}}A^{\frac{1}{2}-\phi}\theta}{\sqrt{|\lambda|}A^\frac{1}{2}u}\\
\nonumber
& &-\dfrac{a^2}{\alpha}\dual{A^\frac{1}{2}v}{A^\frac{1}{2}u} +a^2\dual{A^\frac{1}{2}f^2}{A^\frac{1}{2}u}+\dfrac{a^2}{\alpha\eta}\dual{f^4}{A^{1-\phi}u}\}.
\end{eqnarray}
For $\frac{1}{2}\leq\phi\leq 1$, we have $\frac{1}{2}-\phi\leq\frac{2\phi-1}{2}$, from \eqref{Item04Lemma4} of Lemma \ref{Lemma4},  we have 
\begin{eqnarray}
\label{Eq003Ai}
\hspace*{-0.6cm}\bigg|{\rm Im}\bigg\{-i\dfrac{a^2\beta}{\alpha\eta}\dual{A^\frac{1}{2}\theta}{\lambda A^{\frac{1}{2}-\phi}v} \bigg\} \bigg| 
&\leq & C_\delta\{\|F\|_\mathbb{H}\|U\|_\mathbb{H}+\|F\|^2_\mathbb{H}\}\quad{\rm for}\quad \dfrac{1}{2}\leq \phi\leq 1.
\end{eqnarray}
Besides, using \eqref{Pesp-40} and considering  $\frac{1}{2}\leq \phi\leq \frac{3}{4}$, from estimative \eqref{Item02Lemma4} Lemma \ref{Lemma4}, estimative \eqref{Item03Lemma5A} Lemma \ref{Lemma5A} and estimative \eqref{Item01Lemma4A} Lemma  \ref{Lemma4A},  we have
\begin{eqnarray}
\nonumber
\bigg|{\rm Im}\bigg\{\dfrac{-i}{\eta}\dual{\lambda A^{-\phi} w}{A\theta}\bigg\}\bigg|& \leq & C\{\big |\dual{\lambda A^{-\phi}w}{-i\lambda\theta-\eta A^\phi v-\alpha\eta A^\phi w+f^4}\big |\}\\
\nonumber
&\leq & C\{ |\lambda|^2\|A^{-\phi}w\|^2+|\lambda|^2\|\theta\|^2+\|w\|^2+\|f^4\|^2\}\\
\nonumber
& & +C_\varepsilon|\lambda|\|v\|^2+\varepsilon|\lambda|\|w\|^2\\
\label{Eq004Ai}
&\leq & C_\delta\{\|F\|_\mathbb{H}\|U\|_\mathbb{H}+\|F\|^2_\mathbb{H}\}\quad{\rm for}\quad \dfrac{1}{2}\leq  \phi \leq \dfrac{3}{4}, 
\end{eqnarray}
on the other hand, using \eqref{Pesp-30}, we have
\begin{eqnarray}
\nonumber
\hspace*{-0.4cm}\bigg|{\rm Im}\bigg\{\dfrac{-i}{\eta}\dual{\lambda A^{-\phi} w}{A\theta}\bigg\}\bigg|\hspace*{-0.25cm} &\leq &\hspace*{-0.25cm} C \big|\dual{-a^2Au-a^2\beta Av+\eta A^\phi\theta-w+\alpha f^3}{\lambda A^{-\phi}\theta}   \big|\\
\nonumber
&\leq &\hspace*{-0.35cm} C\big| -a^2\dual{\sqrt{|\lambda|}A^\frac{1}{2}u}{\dfrac{\lambda}{\sqrt{|\lambda|}} A^{\frac{1}{2}-\phi}\theta}-a^2\beta\dual{\lambda A^\frac{2\phi-1}{2}v}{A^{\frac{3}{2}-2\phi}\theta}\\
\nonumber
& & +\eta\lambda\|\theta\|^2-\dual{\sqrt{|\lambda|}w}{\dfrac{\lambda}{\sqrt{|\lambda|}}A^{-\phi}\theta}+\alpha\dual{f^3}{\lambda A^{-\phi}\theta}\big|,
\end{eqnarray}
considering  $\frac{3}{4}\leq\phi\leq 1$, we have $\frac{1}{2}-\phi\leq 0$,  $-\phi<0$, $\frac{3}{2}-2\phi<\frac{1}{2}$ and   $-\phi\leq \phi-\frac{3}{2}$, applying Cauchy-Schwarz and Young inequalities, continuous immersion and estimates \eqref{dis-10}, \eqref{Item01Lemma3} of Lemma \ref{Lemma3}, \eqref{Item02Lemma4A} of Lemma \ref{Lemma4A},  \eqref{Item01Lemma5A} and \eqref{Item04Lemma54} of Lemma \ref{Lemma5A}, for $\varepsilon>0$ exists $C_\varepsilon>0$ such that
\begin{multline}
\label{Eq004Bi}
\bigg|{\rm Im}\bigg\{\dfrac{-i}{\eta}\dual{\lambda A^{-\phi} w}{A\theta}\bigg\}\bigg|\leq  C_\delta\{\|F\|_\mathbb{H}\|U\|_\mathbb{H}+\|F\|^2_\mathbb{H}\}\\+\varepsilon|\lambda|\|w\|^2\quad{\rm for}\quad \dfrac{3}{4}\leq \phi\leq 1.
\end{multline}
And,  from $\frac{1}{2}\leq \phi\leq 1$, we have $1-\phi\leq\frac{1}{2}$, then
\begin{multline}
\label{Eq005Ai}
\big|\frac{1}{\eta}{\rm Im}\dual{f^3}{A^{1-\phi}\theta}\big|\leq C\{\|f^3\|^2+\|A^\frac{1}{2}\theta\|^2\}\\
\leq C_\delta\{\|F\|_\mathbb{H}\|U\|_\mathbb{H}+\|F\|^2_\mathbb{H}\}\quad {\rm for}\quad\dfrac{1}{2}\leq \phi\leq 1.
\end{multline}
Applying Cauchy-Schwarz and Young inequalities in \eqref{Eq002Ai} and using estimates \eqref{Eq004Ai}--\eqref{Eq005Ai}, for $\varepsilon>0$ exists $C_\varepsilon>0$, such that 
\begin{eqnarray*}
|\lambda|\|w\|^2 &\leq & C_\delta\{ \|A^\frac{1}{2}u\|\|A^\frac{1}{2}f^2\|+|\lambda|\|A^\frac{2\phi-1}{2}v\|^2+\|A^{2\phi-\frac{3}{2}}v\|\|A^\frac{1}{2}f^2\|+\|f^4\|\|A^{\phi-1}w\|\\
& &+\|f^3\|\|w\|+\|f^4\|\|A^{1-\phi}v\|^2+\|A^\frac{2\phi-1}{2}w\|^2+\|A^{\frac{3}{2}-2\phi}\theta\|^2+ |\lambda|\|A^{\frac{1}{2}-\phi}\theta\|^2
\\
& &+|\lambda|\|A^\frac{1}{2}u\|^2+\|A^\frac{1}{2}v\|^2 +\|A^\frac{1}{2}u\|^2+\|A^\frac{1}{2}f^2\|\|A^\frac{1}{2}u\|+\|f^4\|\|A^{1-\phi}u\|
\\
& & +C_\varepsilon |\lambda| \| A^\frac{\phi-1}{2}\theta\|^2 +\varepsilon |\lambda|\{  \|A^\frac{\phi-1}{2}w\|^2+\|w\|^2\}+ C_\delta\{\|F\|_\mathbb{H}\|U\|_\mathbb{H}+\|F\|^2_\mathbb{H}\}.
\end{eqnarray*}
For $\frac{1}{2}\leq \phi\leq 1$, we have:  $2\phi-\frac{3}{2}\leq\frac{1}{2}, \;\frac{\phi}{2}-1< \frac{1}{2}$,  $\phi-1\leq\frac{\phi-1}{2}\leq 0<\frac{1}{2}$,  $\frac{3}{2}-2\phi\leq\frac{1}{2}$, \\
$\frac{1}{2}-\phi\leq 0$, $1-\phi\leq \frac{1}{2}$, using continuous immersion, finish proof this item.
\\
{\bf Proof.} $(ii)$: Performing the duality product between $a^2\beta Av$ and $w$ e using  equation \eqref{Pesp-20} and  advantage of the self-adjointness of the powers of the operator A, we obtain
\begin{multline}
\label{Eq000iii}
a^2\beta\dual{Av}{w}=a^2\beta\dual{A^\frac{1}{2}v}{i\lambda A^\frac{1}{2}v-A^\frac{1}{2}f^2}=-ia^2\beta\lambda\|A^\frac{1}{2}v\|^2-a^2\beta\dual{A^\frac{1}{2}v}{A^\frac{1}{2}f^2}.
\end{multline}
On the other hand, multiplying by $w$ equation \eqref{Pesp-30} and by $\eta A^{\phi-1}w$ the equation \eqref{Pesp-40} and adding the results, we have
\begin{multline}
\label{Eq000v}
i\alpha\lambda\|w\|^2+a^2\dual{Au}{i\lambda v-f^2}+a^2\beta\dual{Av}{w}-\eta\dual{A^\phi\theta}{w}\\
+\|w\|^2+i\eta\dual{\lambda\theta}{A^{\phi-1}w}+\eta^2\dual{A^\phi v}{i\lambda A^{\phi-1}v-A^{\phi-1}f^2}\\
+\alpha\eta^2\|A^\frac{2\phi-1}{2}w\|^2+\eta\dual{A^\phi\theta}{w}=\alpha\dual{f^3}{w}+\eta\dual{f^4}{A^{\phi-1}w}.
\end{multline}
 Taking imaginary part in \eqref{Eq000v} and using \eqref{Eq000iii}, we have
\begin{multline}
\label{Eq001ii}
\alpha\lambda\|w\|^2=a^2\beta\lambda\|A^\frac{1}{2}v\|^2+\eta^2\lambda\|A^\frac{2\phi-1}{2}v\|^2 +{\rm Im}\{a^2\beta\dual{A^\frac{1}{2}v}{A^\frac{1}{2}f^2}\\
+ia^2\dual{\sqrt{|\lambda|}A^\frac{1}{2}u}{\dfrac{\lambda}{\sqrt{|\lambda|}} A^\frac{1}{2}v}
+a^2\dual{A^\frac{1}{2}u}{A^\frac{1}{2} f^2}+\eta^2\dual{A^{2\phi-\frac{3}{2}}v}{A^\frac{1}{2}f^2}\\
-i\eta\dual{\dfrac{\lambda}{\sqrt{|\lambda|}}\theta}{\sqrt{|\lambda|}A^{\phi-1}w} +\alpha\dual{f^3}{w}+\eta\dual{f^4}{A^{\phi-1}w}\}.
\end{multline}
Then
\begin{multline}
\label{Eq002ii}
a^2\beta\lambda\|A^\frac{1}{2}v\|^2+\eta^2\lambda\|A^\frac{2\phi-1}{2}v\|^2=\alpha\lambda\|w\|^2-{\rm Im}\{a^2\beta\dual{A^\frac{1}{2}v}{A^\frac{1}{2}f^2}\\
+ia^2\dual{\sqrt{|\lambda|}A^\frac{1}{2}u}{\dfrac{\lambda}{\sqrt{|\lambda|}} A^\frac{1}{2}v}
+a^2\dual{A^\frac{1}{2}u}{A^\frac{1}{2} f^2}+\eta^2\dual{A^{2\phi-\frac{3}{2}}v}{A^\frac{1}{2}f^2}\\
-i\eta\dual{\dfrac{\lambda}{\sqrt{|\lambda|}}\theta}{\sqrt{|\lambda|}A^{\phi-1}w} +\alpha\dual{f^3}{w}+\eta\dual{f^4}{A^{\phi-1}w}\}.
\end{multline}
Therefore, of \eqref{Eq002ii}, as for $\frac{1}{2}\leq\phi\leq 1$ we have $\frac{2\phi-1}{2}\leq\frac{1}{2}$,  $2\phi-\frac{3}{2}\leq \frac{1}{2}$, $\phi-1\leq 0$  and $\phi-1<\frac{\phi-1}{2}$,  applying continuous immersion,   Cauchy-Schwarz and Young inequalities, for $\varepsilon>0$ exists $C_\varepsilon>0$, finish proof this item. 
\end{proof}
\begin{theorem}\label{AnaliticidadeVIBRA}
The semigroup $S(t)=e^{t\mathbb{B}}$  is analytic for $\phi=1$.
\end{theorem}
 \begin{proof}
 From Lemma\eqref{EImaginary}, \eqref{RBEixoIm} is verified.  Therefore, it remains to prove \eqref{Analyticity}, for that it is enough to show that,  let $\delta>0$. There exists a constant $C_\delta>0$ such that the solutions of the system \eqref{Eq1.1}-\eqref{Eq1.3} for $|\lambda|>\delta$,  satisfy the inequality 
 \begin{equation}\label{EquiAnalyticity}
 |\lambda|\dfrac{\|U\|_\mathbb{H}}{\|F\|_\mathbb{H}}\leq C_\delta\quad\Longleftrightarrow\quad |\lambda|\|U\|^2_\mathbb{H}\leq C_\delta\|F\|_\mathbb{H}\|U\|_\mathbb{H}<\infty.
 \end{equation}

Finally,   considering $\phi=1$  and  using   item $(i)$ the Lemmas \ref{Lemma3}, \ref{Lemma4A} and  \ref{Lemma5A}, Lemma \ref{Lemma4N} and item $(iii)$ estimative \eqref{Item01Lemma4N} of Lemma \ref{Lemma4N}, finish the proof of this theorem.
\end{proof}  
\subsection{Gevrey Class}
\label{4.2}
Before exposing our results, it is useful to recall the next definition and result  presented in \cite{Tebou2020} (adapted from
\cite{TaylorM}, Theorem 4, p. 153]).

\begin{definition}\label{Def1.1Tebou} Let $t_0\geq 0$ be a real number. A strongly continuous semigroup $S(t)$, defined on a Banach space $ \mathbb{H}$, is of Gevrey class $s > 1$ for $t > t_0$, if $S(t)$ is infinitely differentiable for $t > t_0$, and for every compact set $K \subset (t_0,\infty)$ and each $\mu > 0$, there exists a constant $ C = C(\mu, K) > 0$ such that
    \begin{equation}\label{DesigDef1.1}
    ||S^{(n)}(t)||_{\mathcal{L}( \mathbb{H})} \leq  C\mu ^n(n!)^s,  \text{ for all } \quad t \in K, n = 0,1,2...
    \end{equation}
\end{definition}
\begin{theorem}[\cite{TaylorM}]\label{Theorem1.2Tebon}
    Let $S(t)$  be a strongly continuous and bounded semigroup on a Hilbert space $ \mathbb{H}$. Suppose that the infinitesimal generator $\mathbb{B}$ of the semigroup $S(t)$ satisfies the following estimate, for some $0 < \Psi < 1$:
    \begin{equation}\label{Eq1.5Tebon2020}
    \lim\limits_{|\lambda|\to\infty} \sup |\lambda |^\Psi ||(i\lambda I-\mathbb{B})^{-1}||_{\mathcal{L}( \mathbb{H})} < \infty.
    \end{equation}
    Then $S(t)$  is of Gevrey  class  $s$   for $t>0$,  for every   $s >\dfrac{1}{\Psi}$.
\end{theorem}
Our main result in this subsection is as follows:
\begin{theorem} \label{TGevreyC}
Let  $S(t)=e^{t\mathbb{B}}$  strongly continuos-semigroups of contractions on the Hilbert space $ \mathbb{H}$, the semigroups $S(t)$ is of Gevrey class $s_i$  $i=1,2$,  for every $s_1> \frac{1}{\Psi_1(\phi)}=\frac{1}{2}$ for $\phi\in (\frac{1}{2},\frac{2}{3})$ and $s_2>\frac{1}{\Psi_2(\phi)}=\frac{\phi}{2\phi-1}$ for $\phi\in[\frac{2}{3},1)$,  such that we have the resolvent estimative:
  \begin{equation}\label{Gevrey01}
(i)\quad   \limsup_{|\lambda|\to\infty} |\lambda |^\frac{1}{2} ||(i\lambda I-\mathbb{B})^{-1}||_{\mathcal{L}( \mathbb{H})} <\infty, \quad {\rm for}\quad \frac{1}{2}<\phi\leq\frac{2}{3},
    \end{equation}
    \begin{equation}\label{Gevrey02}
(ii)\quad   \limsup_{|\lambda|\to\infty} |\lambda |^\frac{2\phi-1}{\phi} ||(i\lambda I-\mathbb{B})^{-1}||_{\mathcal{L}( \mathbb{H})} <\infty, \quad {\rm for}\quad \frac{2}{3}\leq\phi< 1.
 \end{equation}
\end{theorem}
\begin{proof} $(i)$: Note that the estimate   
\begin{equation}\label{EqEquiv1.6TS1}
|\lambda|^\frac{1}{2} ||(i\lambda I-\mathbb{B})^{-1}F||_\mathbb{H}=|\lambda|^\frac{1}{2}\|U\|_\mathbb{H}\leq C_\delta\|F\|_\mathbb{H}\quad{\rm for}\quad \dfrac{1}{2}\leq \phi\leq 1
\end{equation}
 implies the inequality \eqref{Gevrey01}.  Therefore from now on we will show \eqref{EqEquiv1.6TS1},  for this purpose let us estimate the term $|\lambda|\|U\|^2_\mathbb{H}$:  Using  estimative \eqref{Item01Lemma4N} Lemma \ref{Lemma4} and \eqref{Item01Lemma5A} of Lemma \ref{Lemma5A} in \eqref{Item02Lemma4N}, we have
  \begin{equation}
  \label{EstimaLA12V}
  |\lambda|\|A^\frac{1}{2}v\|^2\leq C_\delta\{\|F\|_\mathbb{H}\|U\|_\mathbb{H}+\|F\|^2_\mathbb{H}\}\qquad{\rm for}\qquad \dfrac{1}{2}\leq\phi\leq 1.
  \end{equation}
  Using estimates \eqref{Item01Lemma3}, \eqref{Item01Lemma4N}, \eqref{Item01Lemma5A} and \eqref{EstimaLA12V}, for $\varepsilon>0$ exists $C_\varepsilon>0$,  such that
  \begin{equation*}
  |\lambda|\|U\|^2_\mathbb{H}\leq C_\varepsilon\|F\|^2_\mathbb{H}+\varepsilon\|U\|^2\qquad{\rm for}\qquad \dfrac{1}{2}\leq \phi\leq 1.
  \end{equation*}
  Hence, in particular for $\phi\in (\frac{1}{2},\frac{3}{2}]$, we arrive at
  \begin{equation}
  \label{Gevrey01A}
  |\lambda|^\frac{1}{2}\|U\|_\mathbb{H}\leq C_\varepsilon\|F\|_\mathbb{H}\qquad{\rm for}\qquad \dfrac{1}{2}< \phi\leq \dfrac{2}{3}.
  \end{equation}
 {\bf Proof:} $(ii)$:  We assume $\lambda\in \mathbb{R}$ with  $|\lambda|\geq 1$, we shall borrow some ideas from \cite{LiuR95}.   Let us decompose $w$ as $w=w_1+w_2$,  such that
  \begin{multline}\label{Eq110AnalyS1}
    i\alpha\lambda w_1+Aw_1=\alpha f^3 \quad{\rm in}\quad \mathfrak{D}(A^0)\qquad{\rm and}\\  i\alpha \lambda w_2=-a^2Au-a^2\beta Av+\eta A^\phi\theta-w+Aw_1\quad{\rm in}\quad   \mathfrak{D}(A^0).
    \end{multline}
    Firstly,  applying the product duality  the first equation in \eqref{Eq110AnalyS1}$_1$ by $w_1$, then by $Aw_1$,   we have 
    \begin{eqnarray}\label{Eq112AnalyS1}
    |\lambda|\|w_1\| +|\lambda|^\frac{1}{2}\|A^\frac{1}{2}w_1\|+\|Aw_1\|\leq C\|F\|_\mathbb{H}.
    \end{eqnarray}
  As for $\frac{1}{2}\leq\phi\leq 1$  we have: $0\leq\frac{2\phi-1}{2}\leq\frac{1}{2}$ and   from $A^\frac{2\phi-1}{2}w_2=A^\frac{2\phi-1}{2}w-A^\frac{2\phi-1}{2}w_1$,   using estimative \eqref{Eq112AnalyS1} and  estimative \eqref{Item01Lemma4} Lemma \ref{Lemma4},  we have
\begin{eqnarray*}
\|A^\frac{2\phi-1}{2}w_2\|^2  & \leq & C_\delta\{ \|A^\frac{2\phi-1}{2}w\|^2+\|A^\frac{2\phi-1}{2}w_1\|^2 \}\\
\nonumber
&\leq & C_\delta\{ \|F\|_\mathbb{H}\|U\|_\mathbb{H}+|\lambda|^{-1}\|F\|^2_\mathbb{H}\}, 
\end{eqnarray*}
as for $\frac{1}{2}\leq \phi \leq 1$, we have $-1\leq\frac{1-2\phi}{\phi}$, then
\begin{multline}
\label{Eq000Gevrey2S}
\|A^\frac{2\phi-1}{2}w_2\|^2 \leq  C_\delta\{ \|F\|_\mathbb{H}\|U\|_\mathbb{H}+|\lambda|^\frac{1-2\phi}{\phi}\|F\|^2_\mathbb{H}\} \\
\leq C_\delta|\lambda|^\frac{1-2\phi}{\phi}\{|\lambda|^\frac{2\phi-1}{\phi}\|F\|_\mathbb{H}\|U\|_\mathbb{H}+\|F\|^2_\mathbb{H}\}\quad  {\rm for}\quad \frac{1}{2}\leq \phi \leq 1.
\end{multline}
Now  applying the operator $A^{-\frac{1}{2}}$ in the second equation of  \eqref{Eq110AnalyS1} and using     Lemmas \eqref{Lemma3} and \eqref{Lemma4},  estimates \eqref{dis-10}    and  \eqref{Eq112AnalyS1} and as $-1\leq\frac{1-2\phi}{\phi}$,  we have
    \begin{eqnarray*}
     |\lambda|^2\|A^{-\frac{1}{2}}w_2\|^2  &\leq & C\{ \|A^\frac{1}{2}u\|^2+\|A^\frac{1}{2}v\|^2+\|A^{\phi-\frac{1}{2}}\theta\|^2+\|A^{-\frac{1}{2}}w\|^2\}
+\|A^\frac{1}{2}w_1\|^2\\
& \leq & C_\delta|\lambda|^\frac{1-2\phi}{\phi}\{|\lambda|^\frac{2\phi-1}{\phi}\|F\|_\mathbb{H}\|U\|_\mathbb{H}+\|F\|^2_\mathbb{H}\}\quad  {\rm for}\quad \frac{1}{2}\leq \phi \leq 1.
\end{eqnarray*}
Therefore
\begin{equation}
\label{Eq01GevreyS1}
\|A^{-\frac{1}{2}}w_2\|^2\leq C_\delta |\lambda|^{-[\frac{4\phi-1}{\phi}]}\{|\lambda|^\frac{2\phi-1}{\phi}\|F\|_\mathbb{H}\|U\|_\mathbb{H}+\|F\|^2_\mathbb{H}\}\qquad{\rm for} \qquad\dfrac{1}{2}\leq\phi\leq 1.
    \end{equation}
 By interpolations inequality Theorem \ref{Lions-Landau-Kolmogorov}, since  $-\frac{1}{2}< 0\leq\frac{2\phi-1}{2}$,  using  estimates   \eqref{Eq000Gevrey2S} and \eqref{Eq01GevreyS1},   we derive
 \begin{multline}
 \label{Eq01GevreyS1A}
 \|w_2\|^2 \leq  C(\|A^{-\frac{1}{2}}w_2\|^2)^\frac{2\phi-1}{2\phi}(\|A^\frac{2\phi-1}{2}w_2\|^2)^\frac{1}{2\phi}\\
 \leq C_\delta |\lambda|^{-\big[\big(\frac{4\phi-1}{\phi}\big)\big(\frac{2\phi-1}{2\phi}\big)+\big( \frac{2\phi-1}{\phi}\big)\big(\frac{1}{2\phi}\big)\big)\big]}\big[|\lambda|^\frac{2\phi-1}{\phi}\|F\|_\mathbb{H}\|U\|_\mathbb{H}+\|F\|^2_\mathbb{H}\big]\\
 \leq  C_\delta |\lambda|^\frac{2-4\phi}{\phi} \big[|\lambda|^\frac{2\phi-1}{\phi}\|F\|_\mathbb{H}\|U\|_\mathbb{H}+\|F\|^2_\mathbb{H}\big]  \quad{\rm for}\quad  \dfrac{1}{2}\leq \phi\leq 1.
 \end{multline}
 On the other hand,  as $\|w\|^2 \leq C\{\|w_1\|^2+\|w_2\|^2\}$ and as for $\frac{1}{2}\leq \phi\leq 1$ we have  $-2\leq \frac{2-4\phi}{\phi}$,  using estimates \eqref{Eq112AnalyS1} and \eqref{Eq01GevreyS1A},  we get
\begin{equation*}
\|w\|^2\leq C_\delta|\lambda|^\frac{2-4\phi}{\phi}\big[ |\lambda|^\frac{2\phi-1}{\phi}\|F\|_\mathbb{H}\|U\|_\mathbb{H}+\|F\|^2_\mathbb{H}\big ]      \quad{\rm for }\quad \dfrac{1}{2}\leq \phi\leq 1,
\end{equation*} 
 equivalently
 \begin{eqnarray}
  \label{Eq02GevreyS1}
 |\lambda|\|w\|^2 &\leq &C_\delta|\lambda|^{\frac{2-3\phi}{\phi}} \big[ |\lambda|^\frac{2\phi-1}{\phi}\|F\|_\mathbb{H}\|U\|_\mathbb{H}+\|F\|^2_\mathbb{H}\big ]      \quad{\rm for }\quad \dfrac{1}{2}\leq \phi\leq 1.
 \end{eqnarray}
 Now applying the estimate \eqref{Eq02GevreyS1} in   estimative 
  \eqref{Item02Lemma4N} of Lemma \ref{Lemma4N} for $\frac{1}{2}\leq \phi\leq 1$,  we get
  \begin{multline}
\label{Eq03GevreyS1}
|\lambda|\|A^\frac{1}{2}v\|^2\leq C_\delta \{ |\lambda| \|\theta\|^2 +|\lambda|^\frac{2-3\phi}{\phi}\{|\lambda|^\frac{2\phi-1}{\phi}\|F\|_\mathbb{H}\|U\|_\mathbb{H}+\|F\|^2_\mathbb{H}\}+\|F\|_\mathbb{H}\|U\|_\mathbb{H}\}\\
+C_\varepsilon|\lambda| \|A^\frac{1}{2}u\|^2+\varepsilon|\lambda|\|A^\frac{1}{2}v\|^2 \quad{\rm for}\quad \dfrac{1}{2}\leq\phi\leq 1.
\end{multline}
 for other hand, using \eqref{Item01Lemma3} and \eqref{Item01Lemma5A}, Lemmas \ref{Lemma3},  \ref{Lemma5A} respectively,  we get
 \begin{equation}\label{EstLambdAmedioV}
 |\lambda|\|A^\frac{1}{2}v\|^2\leq |\lambda|^\frac{2-3\phi}{\phi}\{|\lambda|^\frac{2\phi-1}{\phi}\|F\|_\mathbb{H}\|U\|_\mathbb{H}+\|F\|^2_\mathbb{H}\}\quad{\rm for}\quad \dfrac{1}{2}\leq \phi\leq 1.
 \end{equation}
  On the other hand,  from estimative \eqref{Item02Lemma3} of  Lemma \ref{Lemma3} and 
  \begin{multline*}
a^2|\lambda|^\frac{4\phi-2}{\phi}\|A^\frac{1}{2}u\|^2\leq a^2|\lambda|^2\|A^\frac{1}{2}u\|^2\leq C_\delta\{\|F\|_\mathbb{H}\|U\|_\mathbb{H}+\|F\|^2\}
  \\
 \leq  C_\delta\{|\lambda|^\frac{2\phi-1}{\phi}\|F\|_\mathbb{H}\|U\|_\mathbb{H}+\|F\|^2_\mathbb{H}\}\quad{\rm for}\quad 0\leq\phi\leq 1, 
  \end{multline*}
it is equivalent to
  \begin{equation}
  \label{EstLambdAmedioU}
  |\lambda|\|A^\frac{1}{2}u\|^2\leq |\lambda|^\frac{2-3\phi}{\phi}\{|\lambda|^\frac{2\phi-1}{\phi}\|F\|_\mathbb{H}\|U\|_\mathbb{H}+\|F\|^2_\mathbb{H}\}\quad{\rm for}\quad \dfrac{1}{2}\leq \phi\leq 1,
  \end{equation}
  then  from \eqref{EstLambdAmedioV}, \eqref{EstLambdAmedioU} and 
  \begin{equation*}
  a^2 |\lambda|\|A^\frac{1}{2}u+\alpha A^\frac{1}{2} v\|^2\leq a^2|\lambda|\|A^\frac{1}{2}u\|^2+a^2\alpha|\lambda|\|A^\frac{1}{2}v\|^2, 
  \end{equation*}
  we get
  \begin{equation}
  \label{stLambdAmedioUV}
  a^2|\lambda|\|A^\frac{1}{2}u+\alpha A^\frac{1}{2} v\|^2\leq |\lambda|^\frac{2-3\phi}{\phi}\{|\lambda|^\frac{2\phi-1}{\phi}\|F\|_\mathbb{H}\|U\|_\mathbb{H}+\|F\|^2_\mathbb{H}\}\quad{\rm for}\quad \dfrac{1}{2}\leq \phi\leq 1.
  \end{equation}
  Besides, from \eqref{Item01Lemma5A} of Lemma \ref{Lemma5A}
    \begin{multline*}
|\lambda|^\frac{4\phi-2}{\phi}\|\theta\|^2\leq |\lambda|^\frac{3\phi-1}{\phi}\|\theta\|^2\leq C_\delta|\lambda|^\frac{2\phi-1}{\phi}\|F\|_\mathbb{H}\|U\|_\mathbb{H}
  \\
 \leq  C_\delta\{|\lambda|^\frac{2\phi-1}{\phi}\|F\|_\mathbb{H}\|U\|_\mathbb{H}+\|F\|^2_\mathbb{H}\}\quad{\rm for}\quad 0\leq\phi\leq 1.
  \end{multline*}
   Equivalent to
  \begin{equation}
  \label{EstLambdATheta}
  |\lambda|\|\theta \|^2\leq |\lambda|^\frac{2-3\phi}{\phi}\{|\lambda|^\frac{2\phi-1}{\phi}\|F\|_\mathbb{H}\|U\|_\mathbb{H}+\|F\|^2_\mathbb{H}\}\quad{\rm for}\quad \dfrac{1}{2}\leq \phi\leq 1.
  \end{equation}
  From 
  \begin{equation*}
  |\lambda|\|v+\alpha w\|^2\leq |\lambda|\|v\|^2+\alpha|\lambda|\|w\|^2,
  \end{equation*}
 Finally,    adding the estimates \eqref{EstLambdAmedioV}, \eqref{EstLambdAmedioU}, \eqref{stLambdAmedioUV} and considering estimative \eqref {Eq02GevreyS1}, we get
 \begin{equation}
 \label{GevreyNovo}
 |\lambda|^\frac{4\phi-2}{\phi}\|U\|^2_\mathbb{H} \leq C_\delta\{|\lambda|^\frac{2\phi-1}{\phi}\|F\|_\mathbb{H}\|U\|_\mathbb{H}+\|F\|^2_\mathbb{H}\}\quad{\rm for}\quad \dfrac{1}{2}\leq \phi\leq 1.
 \end{equation}
  Therefore, finish the proof of inequality \eqref{EqEquiv1.6TS1}.
  
We finish the proof of this theorem.
\end{proof}
\subsection{Lack  Gevrey Class and  Analiticity}
\begin{theorem}\label{TLackGevrey}
The semigroup $S(t)=e^{t\mathbb{B}}$  does not support Gevrey classes and is not analytic when $\phi\in [0,\frac{1}{2}]$.
\end{theorem}
 \begin{proof}
Since the operator $A$  is strictly positive,  selfadjoint and it has compact resolvent, its spectrum is constituted by positive eigenvalues $(\sigma_n)$ such that $\sigma_n\con \infty$ as $n\con \infty$. For $n\in \N$ we denote with $e_n$ an unitary $D(A^0)$-norm eigenvector associated to the eigenvalue $\sigma_n$, that is,
\begin{equation}\label{auto-10}
Ae_n=\sigma_ne_n,\quad \|e_n\|_{\mathfrak{D}(A^0)}=\|e_n\|=1,\quad n\in\N.
\end{equation}
Considering  the eigenvalues and eigenvectors of the operator $A$ as in \eqref{auto-10}.  Let $F_n=(0,0,-\frac{e_n}{\alpha}, 0 )\in \mathbb{H}$. The solution $U_n=(u_n,v_n,w_n, \theta_n)$ of the system $(i\lambda_n I-\mathbb{B})U_n=F_n$ satisfies $v_n=i\lambda_n u_n$,  $w_n=-\lambda_n^2u_n$,  $A^\psi e_n=\sigma_n^\psi e_n$  for $\psi\in \mathbb{R}$ and the following equations
\begin{eqnarray*}
 -i\alpha\lambda^3_n  u_n+a^2Au_n+ia^2\beta\lambda_n Au_n-\lambda^2_nu_n-\eta A^\phi \theta_n&=& -e_n,\\
i\eta\lambda_n A^\phi u_n-\alpha\eta\lambda^2_nA^\phi u_n+ i\lambda_n\theta_n +A\theta_n& = & 0.
\end{eqnarray*}
Equivalently, 
\begin{eqnarray*}
\{ i\big[ \alpha\lambda^3_n I -a^2\beta\lambda_nA \big]+   \lambda^2_nI-a^2A\}u_n+\eta A^\phi\theta_n&=&e_n,\\
\{i\eta\lambda_nA^\phi -\alpha\eta\lambda^2_n A^\phi\}u_n+ \{ i\lambda_nI +A\}\theta_n& = & 0.
\end{eqnarray*}
 Let us see whether this system admits solutions of the form
 \begin{equation*}
    u_n=\mu_n e_n,\quad \theta_n=\nu_n e_n,
 \end{equation*}
for some complex numbers $\mu_n$ and $\nu_n$. Then, the numbers $\mu_n$, $\nu_n$ should satisfy the algebraic system
\begin{eqnarray}\label{eq01system6}
\{ i\big[ \alpha\lambda^3_n-a^2\beta\lambda_n\sigma_n]+\lambda_n^2- a^2\sigma_n\}\mu_n+\eta \sigma_n^\phi\nu_n&=& 1,\\
 \label{eq02system6}
\{ i\eta\lambda_n \sigma_n^\phi -\alpha\eta\lambda^2_n\sigma_n^\phi\}\mu_n+ \{ i\lambda_n +\sigma_n\}\nu_n& = & 0.
\end{eqnarray}
At this point,  we introduce the numbers
\begin{equation}\label{Numbers6}
\sigma_n:=\bigg[\dfrac{\alpha}{a^2\beta}\bigg]\lambda_n^2. 
\end{equation} 
Then
\begin{multline}\label{MuLambda6}
 |\sigma_n|\approx|\lambda_n|^2, \quad \lambda_n^2-a^2\sigma_n=\bigg[\frac{\beta-\alpha}{\beta}\bigg]\lambda_n^2, \quad \eta\sigma_n^\phi=\bigg[\dfrac{\eta\alpha^\phi}{a^{2\phi}\beta^\phi}\bigg]\lambda_n^{2\phi}\\
   i\eta\lambda_n\sigma_n^\phi-\alpha\eta\lambda_n^2\sigma_n^\phi =i\bigg[\dfrac{\eta\alpha^\phi}{a^{2\phi}\beta^\phi}\bigg]\lambda_n^{1+2\phi}-\bigg[\dfrac{\alpha^{1+\phi}\eta}{a^{2\phi}\beta^\phi}\bigg]\lambda_n^{2+2\phi} \\ {\rm and} \quad i\lambda_n+\sigma_n= i\lambda_n+\bigg[\dfrac{\alpha}{a^2\beta}\bigg]\lambda_n^2.
\end{multline}
Using identities of \eqref{Numbers6} and \eqref{MuLambda6} in system \eqref{eq01system6}-\eqref{eq02system6}, we obtain
\begin{eqnarray}\label{eq03system6}
\bigg\{ \dfrac{\beta-\alpha}{\beta}\lambda_n^2\bigg\}\mu_n+\bigg\{\dfrac{\eta\alpha^\phi}{a^{2\phi}\beta^\phi}\lambda_n^{2\phi}\bigg\}\nu_n&=& 1,\\
 \label{eq04system6}
\bigg\{i\dfrac{\eta\alpha^\phi}{a^{2\phi}\beta^\phi}\lambda_n^{1+2\phi}-\dfrac{\alpha^{1+\phi}\eta}{a^{2\phi}\beta^\phi}\lambda_n^{2+2\phi}\bigg \}\mu_n+ \bigg\{ i\lambda_n+\dfrac{\alpha}{a^2\beta}\lambda_n^2 \bigg \}\nu_n& = & 0.
\end{eqnarray}
Then
\begin{multline}\label{eq05system6}
\Delta=\left| \begin{array}{cc}
  \dfrac{\beta-\alpha}{\beta}\lambda_n^2  & \dfrac{\eta\alpha^\phi}{a^{2\phi}\beta^\phi}\lambda_n^{2\phi}\\
i\dfrac{\eta\alpha^\phi}{a^{2\phi}\beta^\phi}\lambda_n^{1+2\phi}-\dfrac{\alpha^{1+\phi}\eta}{a^{2\phi}\beta^\phi}\lambda_n^{2+2\phi}& 
   i\lambda_n+\dfrac{\alpha}{a^2\beta}\lambda_n^2
\end{array}\right  |\\
=\bigg[ \dfrac{\alpha(\beta-\alpha)}{a^2\beta^2}\bigg]\lambda^4_n+\bigg[\dfrac{\alpha^{2\phi+1}\eta^2}{a^{4\phi}\beta^{2\phi}}\bigg]\lambda_n^{2+4\phi}+i\bigg[ \dfrac{\beta-\alpha}{\beta}\lambda_n^3-\dfrac{\eta^2\alpha^{2\phi}}{a^{4\phi}\beta^{2\phi}}\lambda_n^{4\phi+1}\bigg]
\end{multline}
Therefore: \\
   For $\phi\in [0,\frac{1}{2}]$ and  $\beta-\alpha>0$,  we have
\begin{equation}\label{eq06system6}
|\Delta|\approx
|\lambda_n|^{\max\{4,3\}}=|\lambda_n|^4 \quad{\rm for}\quad    0\leq\phi \leq \dfrac{1}{2}.
\end{equation}
Besides:
\begin{equation}\label{eq15system6}
|\Delta_{\mu_n}|=\bigg|\left| \begin{array}{cc}
 1  & \dfrac{\eta\alpha^\phi}{a^{2\phi}\beta^\phi}\lambda_n^{2\phi}\\
0& 
   i\lambda_n+\dfrac{\alpha}{a^2\beta}\lambda_n^2
\end{array}\right  |\bigg|
=\big  | i\lambda_n+\dfrac{\alpha}{a^2\beta}\lambda_n^2\big |\approx |\lambda_n|^2\quad {\rm for}\quad 0\leq\phi\leq \dfrac{1}{2}.
\end{equation}
Then
\begin{equation}\label{eq16system6}
|\mu_n|=\bigg |\dfrac{ \Delta_{\mu_n}}{\Delta} \bigg|\approx
|\lambda_n|^{-2} \quad {\rm for} \quad  0\leq\phi \leq \dfrac{1}{2}.
\end{equation}
From the definition of $\|U\|_\mathbb{H}$ and estimative \eqref{eq16system6}, for $K>0$ we have
  \begin{equation}\label{eq19system6}
 \|U_n\|_\mathbb{H}\geq K\|A^\frac{1}{2}v_n\|=|\sigma_n|^\frac{1}{2}|i\lambda_n||\mu_n|
  \approx K |\lambda_n|^0 \qquad{\rm for} \qquad  0\leq\phi \leq \dfrac{1}{2}
 \end{equation}
 and 
 \begin{equation}\label{eq20system6}
|\lambda_n| \|U_n\|_\mathbb{H}\geq K |\lambda_n|^1 \qquad {\rm for} \qquad 0\leq\phi \leq \dfrac{1}{2}.
 \end{equation}
\begin{remark}
\begin{enumerate}
\item[I)] If there is a Gevrey class for $S(t)$ when $\phi\in [0,\frac{1}{2}]$, there must exist a $\Psi\in (0,1)$ such that the identity 
 \begin{equation}\label{LackGevrey}
    \lim\limits_{|\lambda_n|\to\infty} \sup |\lambda_n |^\Psi ||(i\lambda I-\mathbb{B})^{-1}||_{\mathcal{L}( \mathbb{H})} < \infty.
    \end{equation}
 is verified.   But,     multiplying both sides of the inequality  \eqref{eq19system6} by $|\lambda_n|^\Psi$  for  $0<\Psi<1$, we have 
  \begin{equation}\label{eq19system6A}
|\lambda_n|^\Psi \|U_n\|_\mathbb{H}\geq K |\lambda_n|^\Psi \qquad{\rm for} \qquad  0\leq\phi \leq \dfrac{1}{2},
 \end{equation}
then,    $|\lambda_n|^\Psi\|U_n\|_\mathbb{H}\to\infty$  approaches infinity as $|\lambda_n|\to\infty$. Therefore   for $\phi$ in $[0, \frac{1}{2}]$ the \eqref{LackGevrey} condition fails.   Consequently the semigroup $S(t)$ does not admit a Gevrey class for $ \phi\in [0,\frac{1}{2}]$.  
 
\item[II)] From inequality  \eqref{eq20system6},  $\|U_n\|_\mathbb{H}\to\infty$  approaches infinity as $|\lambda_n|\to\infty$. Therefore   for $\phi$ in $[0, \frac{1}{2}]$, $S(t)$ does not satisfy the condition \eqref{Analyticity} of Theorem \ref{LiuZAnalyticity}, so $S(t)$ is not analytic when $\phi\in[0,\frac{1}{2}]$.
\end{enumerate}
\end{remark}
\end{proof}

\end{document}